\documentclass[reqno,12pt]{amsart}

\usepackage{graphicx}
\usepackage{amsmath}
\usepackage{mathabx}
\usepackage{amssymb}
\usepackage{dsfont}
\usepackage{enumitem}     

\usepackage[headings]{fullpage}    

\numberwithin{equation}{section}
\def\RR{{\mathbb R}}

\def\eps{\varepsilon}
\def\then{\Rightarrow}

\def\cupp{\mathop{\cup}}

\def\ell{l}

\def\mint{{{\bf-}\!\!\!\!\!\!\hspace{-.1em}\int}}

\def\liminf{\mathop{\underline{\lim}}}

\def\exp{{\rm e}}

\def\endproof{$\blacksquare$}

\newtheorem{theorem}{Theorem}[section]
\newtheorem{lemma}[theorem]{Lemma}

\newtheorem{corollary}[theorem]{Corollary}

\theoremstyle{definition}
\newtheorem{definition}[theorem]{Definition}

\theoremstyle{remark}
\newtheorem{remark}[theorem]{Remark}

\title[Regularity of solutions of one dimensional variational obstacle problems]{On the regularity of solutions of one dimensional variational obstacle problems}

\author{Jean-Philippe Mandallena}
\address{{\rm (Jean-Philippe Mandallena)} UNIVERSITE DE NIMES, Laboratoire MIPA, Site des Carmes, Place Gabriel P\'eri, 30021 N\^\i mes, France.}
\email{jean-philippe.mandallena@unimes.fr}


\begin{document}

\maketitle

\begin{abstract}
We study the regularity of solutions of one dimensional variational obstacle problems in $W^{1,1}$ when the Lagrangian is locally H\"older continuous and globally elliptic. In the spirit of the work of Sychev (\cite{sychev89,sychev91,sychev92}), a direct method is presented for investigating such regularity problems with obstacles. This consists of introducing a general subclass $\mathcal{L}$ of $W^{1,1}$, related in a certain way to one dimensional variational obstacle problems, such that every function of $\mathcal{L}$ has Tonelli's partial regularity, and then to prove that, depending on the regularity of the obstacles, solutions of corresponding variational problems belong to $\mathcal{L}$. As an application of this direct method, we prove that if the obstacles are $C^{1,\sigma}$ then every Sobolev solution has Tonelli's partial regularity.
\end{abstract}

\tableofcontents

\newpage

\section{Introduction}

In this paper we consider one dimensional variational obstacle problems of type
\begin{equation}\label{OVP}
\inf\left\{\mathcal{J}_{L}(u;[a,b]):=\int_a^bL(x,u(x),u^\prime(x))dx: u\in \mathcal{A}_{f,g}\right\}, 
\end{equation}
where $L\in C([a,b]\times\RR\times\RR)$ and $\mathcal{A}_{f,g}\subset W^{1,1}([a,b])$ is given by
$$
\mathcal{A}_{f,g}:=\Big\{u\in W^{1,1}([a,b]):u(a)=A,\ u(b)=B\hbox{ and } f\leq u\leq g\Big\}
$$
with $a,b,A,B\in\RR$ with $a<b$ and $f,g\in W^{1,1}([a,b])$ with $f<g$. Usually, the functions $f$ and $g$ are called obstacles.  The object of the paper is to study the regularity of solutions of one dimensional variational obstacle problems of type \eqref{OVP} when the Lagrangian $L$ is locally H\"older continuous and globally elliptic, see the conditions (H$_1$) and (H$_2$) in Section 2. For this, in the spirit of the work of Sychev (\cite{sychev89,sychev91,sychev92}), we develop a direct method which consists of introducing a general subclass  $\mathcal{L}_\omega(L,K,c,\delta_0)$ of $W^{1,1}([a,b])$, see Definition \ref{Def-of-the-subclass-mathcal{L}}, related to \eqref{OVP}, such that under suitable conditions, see (H$_1$), (H$_2$) and especially (H$^{K,\delta_0}_{\omega})$ in \S2.1, every $u\in \mathcal{L}_\omega(L,K,c,\delta_0)$ has Tonelli's partial regularity, see Definition \ref{TPR-Def} and Theorem \ref{General-Regularity-Theorem} which is the central result of the paper. Then, we prove that if the obstacles $f$ and $g$ are in $C^1([a,b])$ then solutions of  \eqref{OVP} belong to $\mathcal{L}_\omega(L,K,c,\delta_0)$, see Lemma \ref{Prop-Appli1}. From Theorem \ref{General-Regularity-Theorem} and Lemma \ref{Prop-Appli1} we then deduce a regularity result (see Theorem \ref{Main-Appli-Theorem}) for solutions of \eqref{OVP} which says that if $L$ is locally H\"older continuous and globally elliptic and if 
\begin{equation}\label{Cond-for-Reg-Sol}
\lim_{\eps\to 0}\int_0^{\exp\eps}\left[\omega_h\left(\xi+\gamma\sqrt{\xi}\right)\right]^{\theta}{d\xi\over\xi}=0\hbox{ for all }h\in\{f,f^\prime,g,g^\prime\}\hbox{ and all } \gamma,\theta>0,
\end{equation}
where $\omega_h:[0,\infty[\to[0,\infty[$ denotes the modulus of continuity of $h$ and $\exp$ is Napier's number, then every solution of  \eqref{OVP} has Tonelli's partial regularity. In particular,  \eqref{Cond-for-Reg-Sol} holds when the obstacles $f$ and $g$ belong to $C^{1,\sigma}([a,b])$, see Corollary \ref{Main-Coro-Appli-2}.

\medskip

The regularity of solutions of one dimensional variational obstacle problems of type \eqref{OVP} was studied by Sychev in \cite{sychev11}  where it is established, for $L$ locally H\"older continuous and locally elliptic,  that when the obstacles are bounded in $W^{1,\infty}$-norm, solutions exist in the class of Lipschitz functions provided that the obstacles are close, and that if furthermore the obstacles are $C^1$ (resp. $C^{1,\sigma}$) then these Lipschitz solutions are $C^1$ (resp. $C^{1,\sigma}$), see \cite[Theorem 1.1]{sychev11} (resp. \cite[Theorem 1.2]{sychev11}).  But nothing is proved for Sobolev solutions. The results of our paper are contributions in this direction.

\medskip

To complete the introduction, let us mention that Gratwick and Preiss proved in \cite{gratwick-preiss11}, without considering obstacles, that if locally H\"older continuity of $L$ fails, i.e., $L$ is only continuous, then Tonelli's partial regularity does not hold in general. Let us also note that regularity and nonregularity phenomena, depending on H\"older and usual continuity, was known in the context of parametric problems long ago, see Reshetnyak's book \cite[Chapter 6]{Reshetnyak-book82}. We also refer the reader to \cite{GST16} where regularity theory without considering obstacles is developed with singular ellipticity.

\medskip

The plan of the paper is as follows. In the next section we give our main results. The central result of the paper (see Theorem \ref{General-Regularity-Theorem}) is stated  in \S 2.1 and its consequences (see Theorem \ref{Main-Appli-Theorem} and Corollary \ref{Main-Coro-Appli-2}) together with their proofs are in \S2.2.  Finally, Theorem \ref{General-Regularity-Theorem} is proved in Section 4. The proof of Theorem \ref{General-Regularity-Theorem} is based upon a technical lemma (see Lemma \ref{Main-Lemma}) which is a generalization of \cite[Lemma 1.1]{sychev92} to the case of obstacle problems. This technical lemma is proved in Section 3.

\section{Main results}

Let $a,b\in\RR$ with $a<b$ and let  $L\in C([a,b]\times\RR\times\RR)$. In what follows, we consider the following two assumptions:
\begin{enumerate}
\item[(H$_1$)] $L$ is locally H\"older continuous, i.e., for each compact $G\subset [a,b]\times\RR\times\RR$ there exist $C=C(G)>0$ and $\alpha=\alpha(G)>0$ such that 
$$
\left|L(x_1,u_1,v_1)-L(x_2,u_2,v_2)\right|\leq C\left(|x_1-x_2|+|u_1-u_2|+|v_1-v_2|\right)^\alpha
$$
for all $(x_1,u_1,v_1), (x_2,u_2,v_2)\in G$;

\item[(H$_2$)]  $L_{vv}\in C([a,b]\times\RR\times\RR)$ and there exists $\mu>0$ such that $L_{vv}\geq\mu$ everywhere. \\
(Then, we have $L_v\in C([a,b]\times\RR\times\RR)$ and 
$$
L(x,u,v_2)-L(x,u,v_1)-L_v(x,u,v_1)(v_2-v_1)\geq {\mu\over 2}(v_2-v_1)^2
$$ 
for all $(x,u,v_1), (x,u,v_2)\in [a,b]\times\RR\times\RR$.)
\end{enumerate}

\medskip

We begin with a general regularity theorem (see Theorem \ref{General-Regularity-Theorem}) with respect to a subclass of $W^{1,1}([a,b])$ related in a certain way to the one dimensional variational obstacle problem \eqref{OVP}, see Definition \ref{Def-of-the-subclass-mathcal{L}}.

\subsection{A general regularity theorem}  For each $u\in W^{1,1}([a,b])$ and each $s,t\in[a,b]$ with $s<t$, we set
$$
k_u(s,t)={u(s)-u(t)\over s-t}
$$
and we define $u_{s,t}\in W^{1,1}([a,b])$ by
\begin{equation}\label{Def-Of-u-s-t}
u_{s,t}(x):=\left\{
\begin{array}{ll}
u(s)+k_u(s,t)(x-s)&\hbox{ if }x\in]s,t[\\
u(x)&\hbox{ if }x\in[a,b]\setminus]s,t[.
\end{array}
\right.
\end{equation}
Then, for every $s,t\in[a,b]$ with $s< t$, one has  
\begin{equation}\label{der-u(s,t)=k-u}
u^\prime_{s,t}(x)=k_u(s,t)\hbox{ for all }x\in]s,t[.
\end{equation}

\begin{definition}\label{Def-of-the-subclass-mathcal{L}}
Let  $K$ be a compact subset of $[a,b]\times \RR$,  let $\omega:[0,\infty[\times[0,\infty[\to[0,\infty[$ be an increasing function in both arguments such that $\omega(k,0)=0$ for all $k\in[0,\infty[$. Given $c>0$ and $\delta_0>0$, we denote by $\mathcal{L}_{\omega}(L,K,c,\delta_0)$ the class of $u\in W^{1,1}([a,b])$ with the following three properties:
\begin{enumerate}
\item[(A$_1$)] $\displaystyle\big\{(x,u(x)):x\in[a,b]\big\}\subset K$; 
\item[(A$_2$)]  $\displaystyle\mathcal{J}_L(u;[a,b]):=\int_a^b L(x,u(x),u^\prime(x))dx\leq c$; 
\item[(A$_3$)] for every $s,t\in[a,b]$ with $s<t$, if $|s-t|\leq\delta_0$ then 
$$
\mathcal{J}_L(u;[a,b])\leq \mathcal{J}_L(u_{s,t};[a,b])+\omega\left(|k_u(s,t)|,|s-t|\right)|s-t|.
$$
\end{enumerate}
\end{definition}

\begin{remark}
When $\omega\equiv0$, the class $\mathcal{L}_\omega(L,K,c,\delta_0)$ is the one introduced by Sychev in \cite{sychev92} for studying the regularity of solutions of variational problems without obstacles. Thus, in Definition \ref{Def-of-the-subclass-mathcal{L}}, the appearance of the function $\omega$ is related to the obstacles (see Lemma \ref{Prop-Appli1}). 
\end{remark}

\begin{remark}\label{Preliminary-Remark}
If $L$ satisfies (H$_2$) then there exists $N>0$ such that for every $u\in W^{1,1}([a,b])$ satisfying (A$_1$) and (A$_2$) one has 
\begin{equation}\label{Remark1-EqUATiON-1}
\|u^\prime\|_{L^2([a,b])}\leq N.
\end{equation}
 Indeed, there exists $c_1>0$ such that for every $u\in W^{1,1}([a,b])$ satisfying (A$_1$), $|u(x)|\leq c_1$ for all $x\in[a,b]$. As $L$ satisfies (H$_2$) there exists $M>0$ such that $L(x,u,v)\geq {\mu\over 4}v^2$ for all $x\in[a,b]$, all $|u|\leq c_1$ and all $|v|\geq M$, where $\mu>0$ is given by (H$_2$). Fix any $u\in W^{1,1}([a,b])$ satisfying (A$_1$) and (A$_2$). Then 
$
\int_{|u^\prime|\geq M}|u^\prime(x)|^2dx\leq {4\over\mu}c,
$
where $c>0$ is given by (A$_2$), and so $\|u^\prime\|^2_{L^2([a,b])}\leq {4\over\mu}c+M^2(b-a)$ and \eqref{Remark1-EqUATiON-1} follows with $N=\sqrt{{4\over\mu}c+M^2(b-a)}$.
\end{remark}

In what follows, given $c>0$ we consider $\Delta_c\subset]0,1[$ given by 
$$
\Delta_c:=\Big\{\delta_0>0:\delta_0+N\sqrt{\delta_0}\leq 1\Big\}
$$
with $N>0$ given by Remark \ref{Preliminary-Remark}. Furthermore, when (H$_1$) and (H$_2$) hold, we introduce the following assumption:
\begin{enumerate}[leftmargin=14mm]
\item[(H$^{K,\delta_0}_{\omega}$)] $\displaystyle\lim\limits_{\eps\to 0}\int_0^{\exp\eps}\overline{\omega}(k,\xi){d\xi\over\xi}=0$ for all $k\in[0,\infty[$, where  $\overline{\omega}:[0,\infty[\times[0,\infty[\to[0,\infty[$ is given by
\begin{equation}\label{def-of-omega-alpha}
\overline{\omega}(k,\eps):=\left[\sqrt{\widehat{\omega}(k,\eps)}\right]^{\alpha(k)}+\sqrt{\widehat{\omega}(k,\eps)}+\widehat{\omega}(k,\eps)
\end{equation}
with $\widehat{\omega}(k,\eps):=\sqrt{\omega\left(k,\eps+N\sqrt{\eps}\right)}$ and, for each $k\in[0,\infty[$, $\alpha(k)>0$ is the H\"older exponent of $L$, given by (H$_1$), with respect to the compact set $G=K_0\times[-(k+M(k)),k+M(k)]$, where  
\begin{equation}\label{Def-Of-the-compact-K0}
K_0:=\Big\{(t,w)\in[a,b]\times \RR:{\rm dist}\big((t,w);K\big)\leq\delta_0+N\sqrt{\delta_0}\Big\}
\end{equation}
with ${\rm dist}\big((t,w);K\big):=\inf\big\{|x-t|+|u-w|:(x,u)\in K\big\}$ and $M(k)>0$ satisfies the following property:
\begin{equation}\label{Property-of-Constant-M(k)}
|\zeta|\geq M(k)\then{\mu\over 4}\zeta^2-2c(k)\big(|\zeta|+1\big)\geq \omega\big(k,\delta_0+N\sqrt{\delta_0}\big)
\end{equation}
with $\mu>0$ given by (H$_2$) and
\begin{equation}\label{Def-of-Constant-c(k)}
c(k):=\max\left\{\sup\limits_{(t,w,p)\in K_0\times[-k,k]}|L(t,w,p)|,\sup\limits_{(t,w,p)\in K_0\times[-k,k]}|L_v(t,w,p)|\right\}.
\end{equation}
\end{enumerate}

\medskip

Before stating our main result, see Theorem \ref{General-Regularity-Theorem} below, recall that every $u\in W^{1,1}([a,b])$ is uniformly continuous on $[a,b]$ and almost everywhere differentiable in $[a,b]$, i.e., 
$$
\big|[a,b]\setminus \Omega_u\big|=0\hbox{ where }\Omega_u:=\Big\{x\in[a,b]:u\hbox{ is differentiable at }x\Big\}.
$$
(Note also that $W^{1,1}([a,b])=AC([a,b])$ where $AC([a,b])$ is the class of absolutely continuous functions on $[a,b]$, see \cite[Chapter 2]{buttazzo-Giaquinta-hildebrandt98} and the references therein.)
\begin{definition}\label{TPR-Def}
We say that $u\in W^{1,1}([a,b])$ has {\em Tonelli's partial regularity} if the following three conditions hold:
\begin{enumerate}
\item[$\bullet$] $\Omega_u$ is an open subset of $[a,b]$;
\item[$\bullet$] $[a,b]\setminus\Omega_u=\Big\{x\in[a,b]:u^\prime(x)=-\infty\hbox{ or }u^\prime(x)=\infty\Big\}$;
\item[$\bullet$] $u^\prime\in C\big([a,b];[-\infty,\infty]\big)$.
\end{enumerate} 
\end{definition}
We denote the class of $u\in W^{1,1}([a,b])$ such that $u$ has Tonelli's partial regularity by $W^{1,1}_{\rm T}([a,b])$. Here is the main result of the paper.

\begin{theorem}\label{General-Regularity-Theorem}
Let $c>0$ and let $\delta_0\in\Delta_c$. If {\rm (H$_1$)}, {\rm (H$_2$)} and {\rm (H$^{K,\delta_0}_{\omega}$)} hold then
$$
\mathcal{L}_{\omega}(L,K,c,\delta_0)\subset W^{1,1}_{\rm T}([a,b]),
$$
i.e., every $u\in\mathcal{L}_{\omega}(L,K,c,\delta_0)$ has Tonelli's partial regularity.
\end{theorem}

Theorem \ref{General-Regularity-Theorem} can be applied to deal with the regularity of solutions of one dimensional variational obstacle problems of type \eqref{OVP}. 

\subsection{Application to the regularity of solutions of variational obstacle problems} Given $f,g\in W^{1,1}([a,b])$ with $f<g$, we set 
$$
\mathcal{S}_{f,g}:=\Big\{u\in\mathcal{A}_{f,g}:\mathcal{J}_L(u;[a,b])\leq \mathcal{J}_L(v;[a,b])\hbox{ for all }v\in\mathcal{A}_{f,g}\Big\}.
$$
Usually, the functions $f$ and $g$ are called the obstacles. The following lemma makes clear the link between the class  $\mathcal{S}_{f,g}$ of solutions of the variational obstacle problem \eqref{OVP} and the class $\mathcal{L}_{\omega}(L,K,c,\delta_0)$ when the obstacles are $C^1$.
\begin{lemma}\label{Prop-Appli1}
Assume that {\rm (H$_1$)} holds and the obstacles $f$ and $g$ belong to $C^1([a,b])$. Let $K:=\big\{(x,u)\in [a,b]\times \RR:f(x)\leq u\leq g(x)\big\}$, let $u\in W^{1,1}([a,b])$, let $c:=|\mathcal{J}_L(u;[a,b])|+1$ and let $\delta_0>0$. Then
$$
u\in\mathcal{S}_{f,g}\then u\in \mathcal{L}_{\omega}(L,K,c,\delta_0)
$$
with $\omega:[0,\infty[\times[0,\infty[\to[0,\infty[$ given by
\begin{equation}\label{Def-of-omega-for-C1-obstacles}
\omega(k,\eps):=C_0\Big[\big(\omega_f(\eps)+\omega_{f^\prime}(\eps)+k\eps\big)^{\alpha_0}+\big(\omega_g(\eps)+\omega_{g^\prime}(\eps)+k\eps\big)^{\alpha_0}\Big],
\end{equation}
where $\omega_f, \omega_{f^\prime},\omega_g, \omega_{g^\prime}:[0,\infty[\to[0,\infty[$ are the moduli of continuity of $f$, $f^\prime$, $g$ and $g^\prime$ respectively, and $C_0,\alpha_0>0$ are given by {\rm(H$_1$)} with $G=[a,b]\times[-M_1,M_1]\times[-M_2,M_2]$, where $M_1:=\max\left\{\|f\|_\infty,\|g\|_\infty\right\}+M_2(b-a)$ and $M_2:=\max\{\|f^\prime\|_\infty,\|g^\prime\|_\infty\}$.
\end{lemma}
\begin{proof}[\bf Proof of Lemma \ref{Prop-Appli1}]
We only have to prove that  (A$_3$) is satisfied with $\omega$ given by \eqref{Def-of-omega-for-C1-obstacles}. 
\smallskip

\paragraph{\bf Step 1: defining an admissible function with respect to the obstacles} Let $s,t\in[a,b]$ with $s<t$ and $|s-t|\leq\delta_0$. Let $v_{s,t}\in W^{1,1}([a,b])$ be given by
$$
v_{s,t}(x)=\left\{
\begin{array}{ll}
u_{s,t}(x)&\hbox{if }f(x)\leq u_{s,t}(x)\leq g(x)\\
f(x)&\hbox{if }f(x)>u_{s,t}(x)\\
g(x)&\hbox{if }u_{s,t}(x)>g(x)
\end{array}
\right.
$$
where $u_{s,t}$ is defined in \eqref{Def-Of-u-s-t}. Note that  
$$
v_{s,t}(x)=u(x)\hbox{ for all }x\in[a,b]\setminus]s,t[.
$$
because $u\in\mathcal{A}_{f,g}$. Then $v_{s,t}\in\mathcal{A}_{f,g}$ and so
$$
\mathcal{J}_L(u;[a,b])\leq \mathcal{J}_L(v_{s,t};[a,b])
$$
because $u\in\mathcal{S}_{f,g}$. 

\smallskip

\paragraph{\bf Step 2: using condition (H\boldmath$_1$\unboldmath)} Set:
\begin{enumerate}
\item[$\bullet$] $A_f:=\{x\in]s,t[:f(x)>u_{s,t}(x)\};$ 
\item[$\bullet$] $B_g:=\{x\in]s,t[:u_{s,t}(x)>g(x)\}$.
\end{enumerate}
(Note that much of the arguments in the proof rely on both $A_f$ and $B_g$ being non-empty. If one is or both are empty then the arguments simplify.) Then, we have 
\begin{eqnarray}
\mathcal{J}_L(u;[a,b])&\leq& \mathcal{J}_L(u_{s,t};[a,b])+\mathcal{J}_L(v_{s,t};[a,b])-\mathcal{J}_L(u_{s,t};[a,b])\nonumber\\
&=&\mathcal{J}_L(u_{s,t};[a,b])+\mathcal{J}_L(f;A_f)-\mathcal{J}_L(u_{s,t};A_f)+\mathcal{J}_L(g;B_g)-\mathcal{J}_L(u_{s,t};B_g)\nonumber\\
&\leq&\mathcal{J}_L(u_{s,t};[a,b])\nonumber\\
&&+\int_{A_f}\big|L(x,f(x),f^\prime(x))-L(x, u_{s,t}(x),k_u(s,t))\big|dx\nonumber\\
&&+\int_{B_g}\big|L(x,g(x),g^\prime(x))-L(x, u_{s,t}(x),k_u(s,t))\big|dx.\label{OVP-Prop-Eq0}
\end{eqnarray}
Since $A_f$ and $B_g$ are open sets, there exist two disjointed countable sequences $\left\{]\alpha_i,\alpha_{i+1}[\right\}_{i\geq 1}$ and $\left\{]\beta_i,\beta_{i+1}[\right\}_{i\geq 1}$ of open intervals with $s\leq\alpha_i<\alpha_{i+1}\leq t$ and $s\leq\beta_i<\beta_{i+1}\leq t$ for all $i\geq 1$ such that:
\begin{eqnarray}
&& A_f=\cupp\limits_{i\geq 1}]\alpha_i,\alpha_{i+1}[\hbox{ and }f(\alpha_i)=u_{s,t}(\alpha_i)\hbox{ for all }i\geq 1;\label{OVP-Prop-Eq2}\\
&& B_g=\cupp\limits_{i\geq 1}]\beta_i,\beta_{i+1}[\hbox{ and }g(\beta_i)=u_{s,t}(\beta_i)\hbox{ for all }i\geq 1.\label{OVP-Prop-Eq3}
\end{eqnarray}
By using Lagrange's finite-increment theorem, we can assert that there exists two sequences $\{x_i\}_{i\geq 1}$ and $\{y_i\}_{i\geq 1}$ with $\alpha_i<x_i<\alpha_{i+1}$ and $\beta_i<y_i<\beta_{i+1}$ such that:
\begin{eqnarray}
&& k_u(s,t)=f^\prime(x_i)\hbox{ for all }i\geq 1;\label{OVP-Prop-Eq4}\\
&& k_u(s,t)=g^\prime(y_i)\hbox{ for all }i\geq 1.\label{OVP-Prop-Eq5}
\end{eqnarray}
For $h\in C([a,b])$ we set $\|h\|_\infty:=\sup\left\{|h(x)|:x\in[a,b]\right\}$. From the above it follows that:
\begin{eqnarray}
&& |u_{s,t}(x)|\leq \max\left\{\|f\|_\infty,\|g\|_\infty\right\}+\|f^\prime\|_\infty(b-a)\hbox{ for all }x\in A_f;\label{OVP-Prop-Eq6}\\
&&  |u_{s,t}(x)|\leq \max\left\{\|f\|_\infty,\|g\|_\infty\right\}+\|g^\prime\|_\infty(b-a)\hbox{ for all }x\in B_g.\label{OVP-Prop-Eq7}
\end{eqnarray}
Indeed, given $x\in A_f$ there exists $i\geq 1$ such that $x\in]\alpha_i,\alpha_{i+1}[$. Hence, using \eqref{OVP-Prop-Eq4} we have 
\begin{eqnarray*}
|u_{s,t}(x)|&\leq& |u(s)|+|k_u(s,t)||x-s|\\
&\leq& |u(s)|+|f^\prime(x_i)|(b-a) \\
&\leq& |u(s)|+\|f^\prime\|_\infty(b-a),
\end{eqnarray*}
and \eqref{OVP-Prop-Eq6} follows because $f\leq u\leq g$. By using the same reasoning with \eqref{OVP-Prop-Eq5} instead of \eqref{OVP-Prop-Eq4} we obtain \eqref{OVP-Prop-Eq7}. Moreover, it is easy to see that:
\begin{enumerate}
\item[$\bullet$] $x\in[a,b]$;
\item[$\bullet$]  $|f(x)|\leq \|f\|_\infty$, $|f^\prime(x)|\leq \|f^\prime\|_\infty$, $|g(x)|\leq \|g\|_\infty$ and $|g^\prime(x)|\leq \|g^\prime\|_\infty$ for all $x\in[a,b]$;
\item[$\bullet$] $|k_u(s,t)|\leq\max\{\|f^\prime\|_\infty,\|g^\prime\|_\infty\}$ by \eqref{OVP-Prop-Eq4} and \eqref{OVP-Prop-Eq5},
\end{enumerate}
and consequently, we have:
\begin{enumerate}
\item[$\bullet$]  $(x,f(x),f^\prime(x))\in[a,b]\times[-M_1,M_1]\times[-M_2,M_2]\hbox{ for all }x\in A_f;$
\item[$\bullet$] $(x,g(x),g^\prime(x))\in[a,b]\times[-M_1,M_1]\times[-M_2,M_2]\hbox{ for all }x\in B_g;$
\item[$\bullet$] $(x,u_{s,t}(x),k_u(s,t))\in [a,b]\times[-M_1,M_1]\times[-M_2,M_2]\hbox{ for all }x\in A_f\cup B_g$.
\end{enumerate}
But, using (H$_1$) we can assert that 
$$
|L(x_1,u_1,v_1)-L(x_2, u_2,v_2)|\leq C_0\big(|x_1-x_2|+|u_1-u_2|+|v_1-v_2|\big)^{\alpha_0}
$$
for all $(x_1,u_1,v_1), (x_2, u_2,v_2)\in [a,b]\times[-M_1,M_1]\times[-M_2,M_2]$, and so:
\begin{enumerate}
\item[$\bullet$] $\big|L(x,f(x),f^\prime(x))-L(x, u_{s,t}(x),k_u(s,t))\big|\leq C_0\big(|f(x)-u_{s,t}(x)|+|f^\prime(x)-k_u(s,t)|\big)^{\alpha_0}$ for all $x\in A_f$;
\item[$\bullet$] $\big|L(x,g(x),g^\prime(x))-L(x, u_{s,t}(x),k_u(s,t))\big|\leq C_0\big(|g(x)-u_{s,t}(x)|+|g^\prime(x)-k_u(s,t)|\big)^{\alpha_0}$ for all $x\in B_g$.
\end{enumerate}
Fix any $x\in A_f$. Then, by \eqref{OVP-Prop-Eq2} there exists $i\geq 1$ such that $x\in]\alpha_i,\alpha_{i+1}[$ and, since  $f(\alpha_i)=u_{s,t}(\alpha_i)$,  we see that
\begin{eqnarray*}
|f(x)-u_{s,t}(x)|&\leq& |f(x)-f(\alpha_i)|+|f(\alpha_i)-u_{s,t}(x)|\\
&=& |f(x)-f(\alpha_i)|+|u_{s,t}(\alpha_i)-u_{s,t}(x)|\\
&=& |f(x)-f(\alpha_i)|+|k_u(s,t)||x-\alpha_i|\\
&\leq&\omega_f(|s-t|)+|k_u(s,t)||s-t|.
\end{eqnarray*}
 Moreover, by \eqref{OVP-Prop-Eq4} we have 
 $$
 |f^\prime(x)-k_u(s,t)|=|f^\prime(x)-f^\prime(x_i)|\leq \omega_{f^\prime}(|s-t|).
 $$
 Consequently
\begin{eqnarray}\label{OVP-Prop-Eq--1}
\big|L(x,f(x),f^\prime(x))-L(x, u_{s,t}(x),k_u(s,t))\big|\leq \omega_1(|k_u(s,t)|,|s-t|)\hbox{ for all }x\in A_f
\end{eqnarray}
with $\omega_1:[0,\infty[\times[0,\infty[\to[0,\infty[$ given by
\begin{equation}\label{OVP-Prop-Eq--1-def}
\omega_1(k,\eps):=C_0\big(\omega_f(\eps)+\omega_{f^\prime}(\eps)+k\eps\big)^{\alpha_0}.
\end{equation}
In the same manner, by using \eqref{OVP-Prop-Eq3} and \eqref{OVP-Prop-Eq5} instead of \eqref{OVP-Prop-Eq2} and \eqref{OVP-Prop-Eq4}, we obtain
\begin{eqnarray}\label{OVP-Prop-Eq--2}
\big|L(x,g(x),g^\prime(x))-L(x, u_{s,t}(x),k_u(s,t))\big|\leq \omega_2(|k_u(s,t)|,|s-t|)\hbox{ for all }x\in B_g
\end{eqnarray}
with $\omega_2:[0,\infty[\times[0,\infty[\to[0,\infty[$ given by
\begin{equation}\label{OVP-Prop-Eq--2-def}
\omega_2(k,\eps):=C_0\big(\omega_g(\eps)+\omega_{g^\prime}(\eps)+k\eps\big)^{\alpha_0}.
\end{equation}

\smallskip

\paragraph{\bf Step 3: end of the proof} Let $\omega:[0,\infty[\times[0,\infty[\to[0,\infty[$ be defined by
$$
\omega(k,\eps):=\omega_1(k,\eps)+\omega_2(k,\eps).
$$
Then, the function $\omega$ is increasing in both arguments and from \eqref{OVP-Prop-Eq--1-def} and \eqref{OVP-Prop-Eq--2-def} we see that $\omega(k,0)=0$ for all $k\in[0,\infty[$. Moreover, combining \eqref{OVP-Prop-Eq0} with \eqref{OVP-Prop-Eq--1} and \eqref{OVP-Prop-Eq--2} we deduce that
\begin{eqnarray*}
\mathcal{J}_L(u;[a,b])&\leq& \mathcal{J}_L(u_{s,t};[a,b])+\int_{A_f}\omega_1(|k_u(s,t)|,|s-t|)dx+\int_{B_g}\omega_2(|k_u(s,t)|,|s-t|)dx\\
&=&\mathcal{J}_L(u_{s,t};[a,b])+\omega_1(|k_u(s,t)|,|s-t|)|A_f|+\omega_2(|k_u(s,t)|,|s-t|)|B_g|.
\end{eqnarray*}
But $A_f$ and $B_g$ are subsets of $]s,t[$, hence $|A_f|\leq|s-t|$ and $|B_g|\leq|s-t|$, and consequently
\begin{eqnarray*}
\mathcal{J}_L(u;[a,b])&\leq&\mathcal{J}_L(u_{s,t};[a,b])+\big(\omega_1(|k_u(s,t)|,|s-t|)+\omega_2(|k_u(s,t)|,|s-t|)\big)|s-t|\\
&=&\mathcal{J}_L(u_{s,t};[a,b])+\omega(|k_u(s,t)|,|s-t|)|s-t|,
\end{eqnarray*}
which shows that (A$_3$) is verified.
\end{proof}

\begin{remark}
In the proof of Lemma \ref{Prop-Appli1} we have in fact established  that (A$_3$) holds without any restriction that $|s-t|<\delta_0$.
\end{remark}

As a consequence of Theorem \ref{General-Regularity-Theorem} and Lemma \ref{Prop-Appli1}, we have

\begin{theorem}\label{Main-Appli-Theorem}
Assume that {\rm(H$_1$)} and {\rm(H$_2$)} are satisfied and the obstacles $f$ and $g$ belong to $C^1([a,b])$. If \eqref{Cond-for-Reg-Sol} holds then
\begin{equation}\label{TPR-of-Solutions-of-(1.1)}
\mathcal{S}_{f,g}\subset W^{1,1}_T([a,b]),
\end{equation}
i.e., every solution of the variational obstacle problem \eqref{OVP} has Tonelli's partial regularity.
\end{theorem}

\begin{proof}[\bf Proof of Theorem \ref{Main-Appli-Theorem}]
Let $u\in\mathcal{S}_{f,g}$. To show that $u\in W^{1,1}_T([a,b])$ it suffices to prove that (H$^{K,\delta_0}_{\omega}$) holds with $K=\big\{(x,u)\in [a,b]\times \RR:f(x)\leq u\leq g(x)\big\}$, $\delta_0\in \Delta_c$ with $c=|\mathcal{J}_L(u;[a,b])|+1$ and $\omega$ given by \eqref{Def-of-omega-for-C1-obstacles}. (Indeed, from Lemma \ref{Prop-Appli1} we then have $u\in \mathcal{L}_\omega(L,K,c,\delta_0)$ and so $u\in W^{1,1}_T([a,b])$ by Theorem \ref{General-Regularity-Theorem}.) According to \eqref{Def-of-omega-for-C1-obstacles} and \eqref{def-of-omega-alpha} it is easily seen that to verify (H$^{K,\delta_0}_{\omega}$) we only need to establish that for $N>0$ given by Remark \ref{Preliminary-Remark}, one has:
\begin{eqnarray}
&&\lim_{\eps\to0}\int_0^{\exp\eps}\left[\omega_h(\xi+N\sqrt{\xi})\right]^{{1\over 4}\alpha_0\alpha(k)}{d\xi\over\xi}=0\hbox{ for all } k\in[0,\infty[;\label{EqUAtION-Appli-THeoREM-1}\\
&& \lim_{\eps\to0}\int_0^{\exp\eps}\left[\omega_h(\xi+N\sqrt{\xi})\right]^{{1\over 4}\alpha_0}{d\xi\over\xi}=0;\label{EqUAtION-Appli-THeoREM-2} \\
&& \lim_{\eps\to0}\int_0^{\exp\eps}\left[\omega_h(\xi+N\sqrt{\xi})\right]^{{1\over 2}\alpha_0}{d\xi\over\xi}=0;\label{EqUAtION-Appli-THeoREM-3}\\
&&\lim_{\eps\to0}\int_0^{\exp\eps}\left[\xi+N\sqrt{\xi}\right]^\beta{d\xi\over\xi}=0\hbox{ for all }\beta>0\label{EqUAtION-Appli-THeoREM-4}.
\end{eqnarray}
But \eqref{EqUAtION-Appli-THeoREM-1}, \eqref{EqUAtION-Appli-THeoREM-2} and \eqref{EqUAtION-Appli-THeoREM-3} are clearly true by applying \eqref{Cond-for-Reg-Sol} with $\gamma=N$ and respectively $\theta={1\over 4}\alpha_0\alpha(k)$, $\theta={1\over 4}\alpha_0$ and $\theta={1\over 2}\alpha_0$. Moreover, for every $\eps>0$ we have 
\begin{eqnarray*}
\int_0^{\exp\eps}\left[\xi+N\sqrt{\xi}\right]^\beta{d\xi\over\xi}&\leq& C\left(\int_0^{\exp\eps}\xi^{\beta-1}d\xi+N^\beta\int_0^{\exp\eps}\xi^{{1\over 2}\beta-1}d\xi\right)\\
&=&C\left({(\exp\eps)^\beta\over\beta}+{2N^\beta}{(\exp\eps)^{{1\over 2}\beta}\over\beta}\right)
\end{eqnarray*}
with $C>0$ (depending on $\beta$). Hence \eqref{EqUAtION-Appli-THeoREM-4} holds, and the proof is complete.
\end{proof}

\medskip

As a direct consequence of Theorem \ref{Main-Appli-Theorem} we have

\begin{corollary}\label{Main-Coro-Appli-2}
If {\rm(H$_1$)} and {\rm(H$_2$)} are satisfied and the obstacles $f$ and $g$ belong to $C^{1,\sigma}([a,b])$, then \eqref{TPR-of-Solutions-of-(1.1)} holds.
\end{corollary}

\begin{proof}[\bf Proof of Corollary \ref{Main-Coro-Appli-2}]
As the obstacles $f$ and $g$ belong to $C^{1,\sigma}([a,b])$, we have
$$
\omega_h(\xi)\leq c\xi^\sigma
$$
for all $h\in\{f,f^\prime,g,g^\prime\}$, all $\xi\geq 0$ and some $c>0$, and \eqref{Cond-for-Reg-Sol} follows.
\end{proof}

\section{A technical lemma}

In this section we prove the following technical lemma which is a generalization of \cite[Lemma 1.1]{sychev92}. (This lemma plays an essential role in the proof of Theorem \ref{General-Regularity-Theorem}.)

\begin{lemma}\label{Main-Lemma}
Let $c>0$ and let $\delta_0\in\Delta_c$. If {\rm(H$_1$)}, {\rm(H$_2$)} and {\rm (H$^{K,\delta_0}_{\omega}$)} hold, then there exists  
$$
\begin{array}{ccccc}
\delta&:&[0,\infty[\times\left[0,{\delta_0\over \exp}\right]&\to&[0,\infty]\\
&&(k,\eps)&\mapsto&\delta(k,\eps)
\end{array}
$$ 
with the following properties{\rm:}
\begin{enumerate}
\item[\rm (P$_1$)] $\delta$ is increasing in both arguments and $\delta(k,0)=0$ for all $k\in[0,\infty[;$ 
\item[\rm(P$_2$)]  there exist $C:[0,\infty[\to]0,\infty[$ and $\eps_0:[0,\infty[\to\left]0,{\delta_0\over \exp}\right]$ such that
$$
\delta(k,\eps)\leq C(k)\left[\left(\sqrt{\eps}\right)^{{1\over 4}\min\{\alpha(k),\alpha^2(k)\}}+\int_0^{\exp\eps}\overline{\omega}(k+\eta_0,\xi){d\xi\over\xi}\right]
$$
for all $k\in[0,\infty[$ and all $\eps\in[0,\eps_0(k)]$ with $\eta_0>0$ an arbitrary fixed constant and $\overline{\omega}$ given by \eqref{def-of-omega-alpha}, and so
\begin{equation}\label{Fundamental-Limit}
\lim_{\eps\to0}\delta(k,\eps)=0
\end{equation}
for all $k\in[0,\infty[$.
\end{enumerate}
Moreover, for every $u\in\mathcal{L}_{\omega}(L,K,c,\delta_0)$ one has
\begin{enumerate}
\item[\rm (P$_3$)] for every $x_1,x_2\in[a,b]$ with $0<x_2-x_1\leq{\delta_0\over\exp}$, \begin{equation}\label{Regu-Fond-Inequality}
\left|k_u(x_1,x_2)-k_u(s,t)\right|\leq\delta\big(\left|k_u(x_1,x_2)\right|,|x_1-x_2|\big)
\end{equation}
for all $s,t\in[x_1,x_2]$ with $s<t$.
\end{enumerate}
\end{lemma}

\begin{proof}[\bf Proof of Lemma \ref{Main-Lemma}]
For each $(t,y,p)\in[a,b]\times\RR\times\RR$ we define $L_{(t,y,p)}:[a,b]\times\RR\times\RR\to\RR$ by
$$
L_{(t,y,p)}(x,u,v):=L(x,u,v)-L_v(t,y,p)v.
$$
Then, it is easy to see that for every $(t,y,p)\in[a,b]\times\RR\times\RR$ and every $u\in\mathcal{L}_{\omega}(L,K,c,\delta_0)$,
\begin{equation}\label{Obstacle-Problem-Fund-Equation}
\mathcal{J}_{L_{(t,y,p)}}(u;[a,b])\leq \mathcal{J}_{L_{(t,y,p)}}(u_{x_1,x_2};[a,b])+\omega\left(|k_u(x_1,x_2)|,|x_1-x_2|\right)|x_1-x_2|
\end{equation}
for all $x_1,x_2\in[a,b]$ such that $0<x_2-x_1\leq\delta_0$.

\smallskip

\paragraph{\bf Step 1: using (H\boldmath$_1$\unboldmath) and (H\boldmath$_2\unboldmath$)} For each $k\in[0,\infty[$ and each $\eps\in[0,\widehat{\delta}_0]$, where $\widehat{\delta}_0:=\delta_0+N\sqrt{\delta_0}$ with $N>0$ given by Remark \ref{Preliminary-Remark}, we consider $I_{k,\eps}\subset[0,\infty[$ defined as the set of $M\geq 0$ such that for all $x_1,x_2,x_3\in[a,b]$, all $u_1,u_2,u_3\in\RR$ and all $p\in\RR$ with
\begin{equation}\label{First-Set-First-Conditions}
\left\{
\begin{array}{l}
(x_1,u_1)\in K; \\
|x_1-x_i|+|u_1-u_i|\leq\eps\hbox{ for }i=2,3;\\ 
|p|\leq k,
\end{array}
\right.
\end{equation}
and for all $\zeta\in \RR$, one has
\begin{equation}\label{First-SEt-ImplicatIOn-AssERtion}
|\zeta|\geq M\then L_{(x_1,u_1,p)}(x_2,u_2,p+\zeta)-L_{(x_1,u_1,p)}(x_3,u_3,p)\geq{\mu\over 4}\zeta^2+\omega(k,\eps),
\end{equation}
where $\mu>0$ is given by (H$_2$). Let $\Delta_1:[0,\infty[\times\left[0,\widehat{\delta}_0\right]\to[0,\infty[$ be defined by
$$
\Delta_1(k,\eps):=\min I_{k,\eps}.
$$
It is clear that $\Delta_1$ is increasing in both arguments and $\Delta_1(k,0)=0$ for all $k\in[0,\infty[$. We claim that there exist $C_1,\alpha:[0,\infty[\to]0,\infty[$ such that 
\begin{equation}\label{First-Lemma-Set1-EqGoal1}
\Delta_1(k,\eps)\leq \sqrt{2\over\mu}\sqrt{2C_1(k)\eps^{\alpha(k)}+\omega(k,\eps)}
\end{equation}
 for all $k\in[0,\infty[$ and all $\eps\in[0,\widehat{\delta}_0]$. Indeed, fix any $k\in[0,\infty[$ and any $\eps\in[0,\widehat{\delta}_0]$. Let $x_1,x_2,x_3\in[a,b]$,  $u_1,u_2,u_3\in\RR$ and $p\in\RR$ be such that \eqref{First-Set-First-Conditions} is satisfied. First of all, using (H$_2$), for every $\zeta\in\RR$ we have
 \begin{eqnarray*}
L_{(x_1,u_1,p)}(x_2,u_2,p+\zeta)\hskip-1mm-\hskip-1mm L_{(x_1,u_1,p)}(x_3,u_3,p)\hskip-2mm&=&\hskip-2mm L(x_2,u_2,p+\zeta)-L_v(x_1,u_1,p)\zeta-L(x_3,u_3,p)\\
\hskip-2mm&=&\hskip-2mm L(x_2,u_2,p+\zeta)\hskip-1mm-\hskip-1mmL_v(x_2,u_2,p)\zeta\hskip-1mm+\hskip-1mmL_v(x_2,u_2,p)\zeta\\
&&-L_v(x_1,u_1,p)\zeta\hskip-1mm-\hskip-1mmL(x_2,u_2,p)\hskip-1mm+\hskip-1mmL(x_2,u_2,p)\\
&&-L(x_3,u_3,p)\\
\hskip-2mm&\geq&\hskip-2mm L(x_2,u_2,p+\zeta)\hskip-1mm-\hskip-1mmL_v(x_2,u_2,p)\zeta\hskip-1mm-\hskip-1mmL(x_2,u_2,p)\\
&&-\big(|L_v(x_2,u_2,p)|\hskip-1mm+\hskip-1mm|L_v(x_1,u_1,p)|\big)|\zeta|\\
&&-\big(|L(x_2,u_2,p)|\hskip-1mm+\hskip-1mm|L(x_3,u_3,p)|\big)\\
\hskip-2mm&\geq&\hskip-2mm {\mu\over 2}\zeta^2-2c(k)\big(|\zeta|+1\big),
\end{eqnarray*}
where $\mu>0$ is given by (H$_2$) and $c(k)\geq 0$ by \eqref{Def-of-Constant-c(k)}. Consequently, considering $M(k)>0$ with the property \eqref{Property-of-Constant-M(k)} and using the fact that $\omega(k,\cdot)$ is increasing, we can assert that
\begin{equation}\label{NeW-PrOoF-Eq1}
|\zeta|\geq M(k)\then L_{(x_1,u_1,p)}(x_2,u_2,p+\zeta)-L_{(x_1,u_1,p)}(x_3,u_3,p)\geq {\mu\over 4}\zeta^2+\omega(k,\eps).
\end{equation}
Secondly, for every $\zeta\in\RR$ we have
\begin{eqnarray*}
L_{(x_1,u_1,p)}(x_2,u_2,p+\zeta)\hskip-1mm-\hskip-1mmL_{(x_1,u_1,p)}(x_3,u_3,p)\hskip-3mm&\geq&\hskip-3mm-|L_{(x_1,u_1,p)}(x_2,u_2,p+\zeta)\hskip-1mm-\hskip-1mmL_{(x_1,u_1,p)}(x_1,u_1,p+\zeta)|\\
&&\hskip-3mm+L_{(x_1,u_1,p)}(x_1,u_1,p+\zeta)\hskip-1mm-\hskip-1mmL_{(x_1,u_1,p)}(x_1,u_1,p)\\
&&\hskip-3mm-|L_{(x_1,u_1,p)}(x_1,u_1,p)-L_{(x_1,u_1,p)}(x_3,u_3,p)|.
\end{eqnarray*}
Let ${C}_1(k)>0$ and ${\alpha}(k)>0$ be given by (H$_1$) with $G=K_0\times[-(k+M(k)),k+M(k)]$ where $K_0$ is given by \eqref{Def-Of-the-compact-K0}. Then, for every $\zeta\in\RR$ such that $|\zeta|\leq M(k)$ we have
\begin{eqnarray*}
L_{(x_1,u_1,p)}(x_2,u_2,p+\zeta)-L_{(x_1,u_1,p)}(x_3,u_3,p)&\hskip-1mm\geq\hskip-1mm&-2{C}_1(k)\eps^{{\alpha}(k)}\\
&&\hskip-1mm+L_{(x_1,u_1,p)}(x_1,u_1,p+\zeta)-L_{(x_1,u_1,p)}(x_1,u_1,p).
\end{eqnarray*}
On the other hand, by (H$_2$) we have
\begin{eqnarray*}
L_{(x_1,u_1,p)}(x_1,u_1,p+\zeta)-L_{(x_1,u_1,p)}(x_1,u_1,p)\hskip-2mm&=&\hskip-2mmL(x_1,u_1,p+\zeta)-L(x_1,u_1,p)-L_v(x_1,u_1,p)\zeta\\
&\geq&\hskip-2mm{\mu\over 2}\zeta^2,
\end{eqnarray*}
and so
$$
|\zeta|\leq M(k)\then L_{(x_1,u_1,p)}(x_2,u_2,p+\zeta)-L_{(x_1,u_1,p)}(x_3,u_3,p)\geq -2{C}_1(k)\eps^{{\alpha}(k)}+{\mu\over 2}\zeta^2.
$$
Consequently, setting $M_1:=\sqrt{2\over\mu}\sqrt{2C_1(k)\eps^{\alpha(k)}+\omega(k,\eps)}$ we can assert that
\begin{equation}\label{NeW-PrOoF-Eq2}
\big(|\zeta|\hskip-1mm\leq\hskip-1mm M(k)\hbox{ and }|\zeta|\hskip-1mm\geq\hskip-1mm M_1\big)\hskip-1mm\then \hskip-1mm L_{(x_1,u_1,p)}(x_2,u_2,p+\zeta)\hskip-1mm-\hskip-1mm L_{(x_1,u_1,p)}(x_3,u_3,p)\hskip-1mm\geq\hskip-1mm {\mu\over 4}\zeta^2\hskip-1mm+\hskip-1mm\omega(k,\eps).
\end{equation}
Combining \eqref{NeW-PrOoF-Eq1} with \eqref{NeW-PrOoF-Eq2} we see that  \eqref{First-SEt-ImplicatIOn-AssERtion} holds with $M=M_1$, and  \eqref{First-Lemma-Set1-EqGoal1} follows.

For each $k\in[0,\infty[$ and each $\eps\in[0,\widehat{\delta}_0]$  we consider $J_{k,\eps}\subset[0,\infty[$ defined as the set of all $M=|L_{(x_1,u_1,p_1)}(x_2,u_2,p_2)-L_{(x_1,u_1,p_1)}(x_3,u_3,p_3)+\omega(k,\eps)|$ for which $x_1,x_2,x_3\in[a,b]$, $u_1,u_2,u_3\in\RR$ and $p_1,p_2,p_3\in\RR$ are such that
\begin{equation}\label{Def-of-set2-of-First-Lemma-EqUatioN2}
\left\{
\begin{array}{l}
(x_1,u_1)\in K; \\
|x_1-x_i|+|u_1-u_i|\leq\eps\hbox{ for }i=2,3;\\ 
|p_i|\leq k\hbox{ for }i=1,2,3;\\
|p_1-p_i|\leq\Delta_1(k,\eps)\hbox{ for }i=2,3.
\end{array}
\right.
\end{equation}
Let $\Delta_2:[0,\infty[\times\left[0,\widehat{\delta}_0\right]\to[0,\infty[$ be defined by
$$
\Delta_2(k,\eps):=\max J_{k,\eps}.
$$
It is clear that $\Delta_2$ is increasing in both arguments and $\Delta_2(k,0)=0$ for all $k\in[0,\infty[$. We claim that there exists $C_2:[0,\infty[\to]0,\infty[$ such that
\begin{equation}\label{First-Lemma-Goal-EqUAtION2}
\Delta_2(k,\eps)\leq C_1(k)\big(\eps+\Delta_1(k,\eps)\big)^{\alpha(k)}+2C_2(k)\Delta_1(k,\eps)+\omega(k,\eps)
\end{equation}
for all $k\in[0,\infty[$ and all $\eps\in[0,\widehat{\delta}_0]$. Indeed, fix any $k\in[0,\infty[$ and any $\eps\in[0,\widehat{\delta}_0]$. It is easy to see that for each $x_1,x_2,x_3\in[a,b]$, each $u_1,u_2,u_3\in\RR$ and each $p_1,p_2,p_3\in\RR$ satisfying \eqref{Def-of-set2-of-First-Lemma-EqUatioN2}, we have
\begin{eqnarray*}
|L_{(x_1,u_1,p_1)}(x_2,u_2,p_2)-L_{(x_1,u_1,p_1)}(x_3,u_3,p_3)+\omega(k,\eps)|&\leq&|L(x_2,u_2,p_2)-L(x_3,u_3,p_3)|\\
&&+|L_v(x_1,u_1,p_1)||p_2-p_3|\\
&&+\omega(k,\eps).
\end{eqnarray*}
But $|L(x_2,u_2,p_2)-L(x_3,u_3,p_3)|\leq C_1(k)(\eps+\Delta_1(k,\eps))^{\alpha(k)}$ because $L$ is $\alpha(k)$-H\"older continuous on the compact $K_0\times[-(k+M(k)),k+M(k)]$ and $(x_2,u_2,p_2), (x_3,u_3,p_3)\in K_0\times[-k,k]$. Moreover, using the continuity of $L_v$ on the compact $K\times[-k,k]$  we deduce that there exists $C_2(k)>0$ such that $|L_v(x,u,p)|\leq C_2(k)$ for all $(x,u,p)\in K\times[-k,k]$. Consequently, \eqref{First-Lemma-Goal-EqUAtION2} follows since $|p_2-p_3|\leq|p_2-p_1|+|p_1-p_3|\leq 2\Delta_1(k,\eps)$.

\smallskip

\paragraph{\bf Step 2: constructing \boldmath$\delta$\unboldmath\  satisfying (P\boldmath$_1$\unboldmath) and (P\boldmath$_2\unboldmath$)} From \eqref{First-Lemma-Set1-EqGoal1} and \eqref{First-Lemma-Goal-EqUAtION2} it is easily seen that there exists $C_3:[0,\infty[\to]0,\infty[$ such that:
\begin{eqnarray}
&&\hskip-10mm\Delta_1(k,\eps)\leq C_3(k)\left[\left(\sqrt{\eps}\right)^{\alpha(k)}+\sqrt{\omega(k,\eps)}\right];\label{EsTImAtE-1}\\
&&\hskip-10mm\Delta_2(k,\eps)\leq C_3(k)\left[\left(\sqrt{\eps}\right)^{\min\{\alpha(k), \alpha^2(k)\}}+\left(\sqrt{\omega(k,\eps)}\right)^{\alpha(k)}+\sqrt{\omega(k,\eps)}+\omega(k,\eps)\right]\label{EsTImAtE-2}
\end{eqnarray}
for all $k\in[0,\infty[$ and all $\eps\in[0,\widehat{\delta}_0]$. Let $\Delta:[0,\infty[\times[0,\delta_0]\to[0,\infty[$ be defined by
\begin{equation}\label{DefiNItION-Of-DeLtA-for-PrOOf-Of-Lemma-3}
\Delta(k,\eps):=\max\left\{\Delta_1\left(k,\eps+N\sqrt{\eps}\right),{2\over\sqrt{\mu}}\sqrt{\Delta_2\left(k,\eps+N\sqrt{\eps}\right)}\right\}
\end{equation}
(with $N>0$ given by Remark \ref{Preliminary-Remark}). It is clear that $\Delta$ is increasing in both arguments and $\Delta(k,0)=0$ for all $k\in[0,\infty[$. From \eqref{EsTImAtE-1} and \eqref{EsTImAtE-2} we see that there exists $C_4:[0,\infty[\to]0,\infty[$ such that 
\begin{eqnarray}\label{EsTImAtE-3}
\Delta(k,\eps)\leq C_4(k)\left[\left(\sqrt{\eps}\right)^{{1\over 4}\min\{\alpha(k), \alpha^2(k)\}}+\overline{\omega}(k,\eps)\right]
\end{eqnarray}
for all $k\in[0,\infty[$ and all $\eps\in[0,\delta_0]$, where $\overline{\omega}:[0,\infty[\times[0,\infty[\to[0,\infty[$ is given by \eqref{def-of-omega-alpha}.

For each $k\in[0,\infty[$ and each $\eps\in[0,{\delta_0\over \exp}]$ we consider $A_{k,\eps}\subset [0,\infty[$ given by 
$$
A_{k,\eps}:=\left\{\eta\geq 0:4\int_0^{\exp \eps}\Delta(k+\eta,\xi){d\xi\over \xi}\leq\eta\right\}.
$$
Let $\delta:[0,\infty[\times\left[0,{\delta_0\over \exp}\right]\to[0,\infty]$ be defined by
$$
\delta(k,\eps):=\inf A_{k,\eps}.
$$
It is clear that $\delta$ is increasing in both arguments and $\delta(k,0)=0$ for all $k\in[0,\infty[$. So (P$_1$) is satisfied. Fix $\eta_0>0$. Given any $k\in[0,\infty[$, from (H$^{K,\delta_0}_{\omega}$) we have $\lim_{\eps\to0}\int_0^{\exp\eps}\overline{\omega}(k+\eta_0,\zeta){d\zeta\over\zeta}=0$, and so $\lim_{\eps\to0}4\int_0^{\exp\eps}\Delta(k+\eta_0,\zeta){d\zeta\over\zeta}=0$ by using \eqref{EsTImAtE-3}. Hence, there exists $\eps_0(k)\in\left]0,{\delta_0\over\exp}\right]$ such that $4\int_0^{\exp\eps}\Delta(k+\eta_0,\zeta){d\zeta\over\zeta}\leq\eta_0$ for any $\eps\in[0,\eps_0(k)]$. Using the fact that $\Delta(\cdot,\xi)$ is increasing, it follows that 
$$
\Delta\left(k+4\int_0^{\exp\eps}\Delta(k+\eta_0,\zeta){d\zeta\over\zeta},\xi\right)\leq \Delta(k+\eta_0,\xi)
$$
for all $\xi\in\left[0,\exp\eps\right]$, hence 
$$
4\int_0^{\exp\eps}\Delta\left(k+4\int_0^{\exp\eps}\Delta(k+\eta_0,\zeta){d\zeta\over\zeta},\xi\right){d\xi\over\xi}\leq 4\int_0^{\exp\eps}\Delta(k+\eta_0,\xi){d\xi\over\xi},
$$
and consequently $4\int_0^{\exp\eps}\Delta(k+\eta_0,\xi){d\xi\over\xi}\in A_{k,\eps}$. Thus
$$
\delta(k,\eps)\leq 4\int_0^{\exp\eps}\Delta(k+\eta_0,\xi){d\xi\over\xi},
$$ 
for all $k\in[0,\infty[$ and all $\eps\in[0,\eps_0(k)]$, and (P$_2$) follows by using again \eqref{EsTImAtE-3} together with (H$^{K,\delta_0}_{\omega}$). 

\begin{remark}\label{L'InFImUM-est-AttEinT}
From the above we see that, under (H$^{K,\delta_0}_{\omega}$), we have $A_{k,\eps}\not=\emptyset$ for all $k\in[0,\infty[$ and all $\eps\in[0,\eps_0(k)]$.  Given $k\in[0,\infty[$ and $\eps\in[0,\eps_0(k)]$ there exists $\{\delta_n\}_{n\geq 1}\subset A_{k,\eps}$ such that $\lim_{n\to\infty}\delta_n=\delta(k,\eps)$ and $\delta(k,\eps)\leq\delta_n$ for all $n\geq 1$. Hence
\begin{equation}\label{Remark-Proof-MT-EquATiOn1}
4\int_0^{\exp\eps}\Delta(k+\delta_n,\xi){d\xi\over\xi}\leq\delta_n\hbox{ for all }n\geq 1,
\end{equation}
and, since $\Delta(\cdot,\xi)$ is increasing,
\begin{eqnarray}\label{Remark-Proof-MT-EquATiOn2}
&& \Delta(k+\delta(k,\eps),\xi)\leq \Delta(k+\delta_n,\xi)\hbox{ for all }n\geq 1\hbox{ and all } \xi\in[0,\exp\eps].
\end{eqnarray}
Using Fatou's lemma, from \eqref{Remark-Proof-MT-EquATiOn1} we have
$$
4\int_0^{\exp\eps}\liminf_{n\to\infty}\Delta(k+\delta_n,\xi){d\xi\over\xi}\leq\liminf_{n\to\infty}4\int_0^{\exp\eps}\Delta(k+\delta_n,\xi){d\xi\over\xi}\leq\lim_{n\to\infty}\delta_n=\delta(k,\eps).
$$
But, from \eqref{Remark-Proof-MT-EquATiOn2} we see that 
$$
 \Delta(k+\delta(k,\eps),\xi){1\over\xi}\leq\liminf_{n\to\infty}\Delta(k+\delta_n,\xi){1\over\xi}\hbox{ for all }\xi\in]0,\exp\eps],
$$
hence
$$
 4\int_0^{\exp\eps}\Delta(k+\delta(k,\eps),\xi){d\xi\over\xi}\leq4\int_0^{\exp\eps}\liminf_{n\to\infty}\Delta(k+\delta_n,\xi){d\xi\over\xi}\leq \delta(k,\eps).
 $$
Thus $\delta(k,\eps)\in A_{k,\eps}$. Consequently, $\delta(k,\eps)=\inf A_{k,\eps}=\min A_{k,\eps}$ for all $k\in[0,\infty[$ and all $\eps\in[0,\eps_0(k)]$.
\end{remark}

\paragraph{\bf Step 3: proving (P\boldmath$_3$\unboldmath)} Let $u\in\mathcal{L}_{\omega}(L,K,c,\delta_0)$ and let $x_1,x_2\in[a,b]$ be such that $0<x_2-x_1\leq{\delta_0\over\exp}$. For simplicity of notation, set:
$$
\left\{
\begin{array}{l}
\eps:=|x_1-x_2|; \\
u_1:=u(x_1);\\ 
k_1:=k_u(x_1,x_2).
\end{array}
\right.
$$
Let $T\subset[x_1,x_2]$ be given by
$$
T:=\Big\{x\in[x_1,x_2]:\big|u^\prime(x)-k_1\big|\geq\Delta_1\left(|k_1|,\eps+N\sqrt{\eps}\right)\Big\}.
$$
We claim that 
\begin{equation}\label{ClAIm-Final-1}
\int_T|u^\prime(x)-k_1|^2dx\leq {4\over\mu}\Delta_2\left(|k_1|,\eps+N\sqrt{\eps}\right)\eps.
\end{equation}

Indeed, $|x-x_1|\leq\eps$, $|u(x)-u_1|\leq N\sqrt{\eps}$ and $|u_{x_1,x_2}(x)-u_1|\leq N\sqrt{\eps}$ for all $x\in[x_1,x_2]$, hence, using \eqref{First-SEt-ImplicatIOn-AssERtion} and the fact that $\omega(|k_1|,\cdot)$ is increasing, if $x\in T$ then
\begin{eqnarray*}
 L_{(x_1,u_1,k_1)}(x,u(x),u^\prime(x))-L_{(x_1,u_1,k_1)}(x,u_{x_1,x_2}(x),k_1)&\geq&{\mu\over 4}|u^\prime(x)-k_1|^2+\omega(|k_1|,\eps+N\sqrt{\eps})\\
 &\geq&{\mu\over 4}|u^\prime(x)-k_1|^2+\omega(|k_1|,\eps),
\end{eqnarray*}
and so 
\begin{eqnarray}\label{ClAIm-Final-1-Eq1}
&&\hskip-19mm L_{(x_1,u_1,k_1)}(x,u_{x_1,x_2}(x),k_1)- L_{(x_1,u_1,k_1)}(x,u(x),u^\prime(x))+\omega(|k_1|,\eps)\leq- {\mu\over 4}|u^\prime(x)-k_1|^2.
\end{eqnarray}
On the other hand, using again the fact that $\omega(|k_1|,\cdot)$ is increasing, if $x\in[x_1,x_2]\setminus T$ then
\begin{eqnarray}\label{ClAIm-Final-1-Eq2}
&&\hskip-19mm L_{(x_1,u_1,k_1)}(x,u_{x_1,x_2}(x),k_1)- L_{(x_1,u_1,k_1)}(x,u(x),u^\prime(x))+\omega(|k_1|,\eps)\leq \Delta_2(|k_1|,\eps+N\sqrt{\eps}).
\end{eqnarray}
But \eqref{Obstacle-Problem-Fund-Equation} holds because $u\in\mathcal{L}_{\omega}(L,K,c,\delta_0)$, which means that
$$
\int_{x_1}^{x_2}\left(L_{(x_1,u_1,k_1)}(x,u_{x_1,x_2}(x),k_1)- L_{(x_1,u_1,k_1)}(x,u(x),u^\prime(x))+\omega(|k_1|,\eps)\right)dx\geq 0,
$$
and consequently, using \eqref{ClAIm-Final-1-Eq1} and \eqref{ClAIm-Final-1-Eq2} we deduce that
$$
{\mu\over 4}\int_T|u^\prime(x)-k_1|^2dx\leq \Delta_2(|k_1|,\eps+N\sqrt{\eps})\eps,
$$
which gives \eqref{ClAIm-Final-1}. Let $\Omega_{x_1,x_2}\subset [x_1,x_2]$ given by
$$
\Omega_{x_1,x_2}:=\Big\{x\in[x_1,x_2]:|u^\prime(x)-k_1\big|\geq\Delta\big(|k_1|,\eps\big)\Big\}.
$$
Then $\Omega_{x_1,x_2}\subset T$ by definition of $\Delta(|k_1|,\eps)$. From \eqref{ClAIm-Final-1} it follows that
$$
\int_T|u^\prime(x)-k_1|^2dx\leq \Delta^2\big(|k_1|,\eps\big)\eps.
$$
We have thus proved that for each $u\in\mathcal{L}_{\omega}(L,K,c,\delta_0)$ and each $y,z\in[a,b]$ with $0<z-y\leq{\delta_0\over \exp}$, one has the following inequality:
\begin{equation}\label{FundaMental-Estimate-Main-Lemma}
\int_{\Omega_{y,z}}\big|u^\prime(x)-k_u(y,z)\big|^2dx\leq \Delta^2\big(|k_u(y,z)|,|y-z|\big)|y-z|,
\end{equation}
where
\begin{equation}\label{FundaMental-Estimate-Main-Lemma-Set}
\Omega_{y,z}:=\Big\{x\in[y,z]:|u^\prime(x)-k_u(y,z)\big|\geq\Delta\big(|k_u(y,z)|,|y-z|\big)\Big\}.
\end{equation}

\smallskip

Finally, the property (P$_3$) follows from the following auxiliary lemma whose proof can be extracted from \cite{sychev92}. (For the convenience of the reader, the proof of Lemma \ref{Auxiliary-Lemma-For-Main-LeMmA} is given below.)

\begin{lemma}\label{Auxiliary-Lemma-For-Main-LeMmA}
Let $u\in W^{1,1}([a,b])$ be such that \eqref{FundaMental-Estimate-Main-Lemma} is satisfied for all $y,z\in[a,b]$ with $0<z-y\leq{\delta_0\over\exp}$. Then $u$ satisfies {\rm (P$_3$)}.
\end{lemma}

\smallskip

But we have proved that every $u\in\mathcal{L}_{\omega}(L,K,c,\delta_0)$ verified  \eqref{FundaMental-Estimate-Main-Lemma} for all $y,z\in[a,b]$ such that $0<z-y\leq{\delta_0\over\exp}$. Consequently (P$_3$) is satisfied for all $u\in\mathcal{L}_{\omega}(L,K,c,\delta_0)$, and the proof of Lemma \ref{Main-Lemma} is complete.
\end{proof}

\begin{proof}[\bf Proof of Lemma \ref{Auxiliary-Lemma-For-Main-LeMmA}]
Let $x_1,x_2\in[a,b]$ be such that $0<x_2-x_1\leq{\delta_0\over\exp}$ and let $s,t\in]x_1,x_2[$ be such that $s<t$. We have to prove \eqref{Regu-Fond-Inequality}, i.e., 
$$
\left|k_u(x_1,x_2)-k_u(s,t)\right|\leq\delta\big(\left|k_u(x_1,x_2)\right|,|x_1-x_2|\big).
$$
For simplicity of notation, we set $k:=k_u(x_1,x_2)$ and $\eps:=|x_1-x_2|$ and, without loss of generality, we assume that $\delta(|k|,\eps)<\infty$. Let $\{I_n:=[x^n_1,x^n_2]\}_{n\geq 1}$ be a decreasing sequence of intervals such that
\begin{eqnarray}\label{Auxiliary-Lemma-For-Main-LeMmA-EQ1}
&&\left\{\begin{array}{l}I_1=[x_1,x_2]\\ I_{n_0+1}\subset [s,t]\subset I_{n_0}\hbox{ for some }n_0\geq 1,\end{array}\right.
\end{eqnarray}
and, setting $\eps_n:=|x^n_1-x^n_2|$ for all $n\geq 1$,
\begin{eqnarray}\label{Auxiliary-Lemma-For-Main-LeMmA-EQ2}
\left\{\begin{array}{l}\eps_1=\eps\\
\eps_{n+1}={1\over\exp}\eps_{n}\hbox{, i.e., }\exp\eps_{n+1}=\eps_n,\hbox{ for all }n\geq 1.\end{array}\right.
\end{eqnarray}
(From \eqref{Auxiliary-Lemma-For-Main-LeMmA-EQ2} we see that $\eps_n={1\over \exp^{n-1}}\eps$ for all $n\geq 1$, and so $\lim_{n\to\infty}\eps_n=0$ with $\eps_1>\eps_2>\cdots>\eps_n>\eps_{n+1}>\cdots$.)

Taking Remark \ref{L'InFImUM-est-AttEinT} into account and noticing that $\cupp_{n=1}^\infty[\exp\eps_{n+1},\exp\eps_n]=\cupp_{n=1}^\infty[\eps_{n},\exp\eps_n]=]0,\exp\eps]$, we see that \begin{eqnarray*}
\delta(|k|,\eps)\geq 4\int_0^{\exp \eps}\Delta(|k|+\delta(|k|,\eps),\xi){d\xi\over \xi}&=& 4\int_{\cupp\limits_{n=1}^\infty[\eps_{n},\exp\eps_n]}\Delta(|k|+\delta(|k|,\eps),\xi){d\xi\over \xi}\\
&=&4\sum_{n=1}^\infty\int_{\eps_{n}}^{\exp\eps_n}\Delta(|k|+\delta(|k|,\eps),\xi){d\xi\over \xi}.
\end{eqnarray*}
Using the fact that $\Delta:[0,\infty[\times[0,\delta_0]\to[0,\infty[$ defined by \eqref{DefiNItION-Of-DeLtA-for-PrOOf-Of-Lemma-3} is increasing in both arguments, it follows that
\begin{eqnarray}
\delta(|k|,\eps)&\geq&4\sum_{n=1}^\infty\Delta(|k|+\delta(|k|,\eps),\eps_n)\int_{\eps_{n}}^{\exp\eps_n}{d\xi\over \xi}\nonumber\\
&=&4\sum_{n=1}^\infty\Delta(|k|+\delta(|k|,\eps),\eps_n)\big[\ln(\exp\eps_n)-\ln(\eps_{n})\big]\nonumber\\
&=&4\sum_{n=1}^\infty\Delta(|k|+\delta(|k|,\eps),\eps_n).\label{FundamenTal-AssertION-Which-iS-proVed-by-Induction-EQUaTion-00}
\end{eqnarray}
Thus, to prove \eqref{Regu-Fond-Inequality} it is sufficient to establish the following assertion:
\begin{equation}\label{FundamenTal-AssertION-Which-iS-proVed-by-Induction}
\forall m\geq 1\ \mathcal{P}(m)
\end{equation}
with $\mathcal{P}(m)$ given by
$$
\forall x_1\leq \sigma<\tau\leq x_2\  \Big[I_{m+1}\subset[\sigma,\tau]\subset I_m\then|k-k_u(\sigma,\tau)|\leq 4\sum_{n=1}^m\Delta(|k|+\delta(|k|,\eps),\eps_n)\Big].
$$ 
Indeed, applying \eqref{FundamenTal-AssertION-Which-iS-proVed-by-Induction} with $m=n_0$, $\sigma=s$ and $\tau=t$, we have
$$
|k-k_u(s,t)|\leq 4\sum_{n=1}^{n_0}\Delta(|k|+\delta(|k|,\eps),\eps_n),
$$
and \eqref{Regu-Fond-Inequality} follows by using \eqref{FundamenTal-AssertION-Which-iS-proVed-by-Induction-EQUaTion-00} together with the fact that $4\sum_{n=1}^{n_0}\Delta(|k|+\delta(|k|,\eps),\eps_n)\leq 4\sum_{n=1}^{\infty}\Delta(|k|+\delta(|k|,\eps),\eps_n)$.
\end{proof}

\medskip

\paragraph{\bf Proof of \eqref{FundamenTal-AssertION-Which-iS-proVed-by-Induction}} We proceed by induction on $m$. 

\smallskip

\paragraph{\bf Step 1: base case} Assume that $\mathcal{P}(1)$ is false. Then, there exists $\sigma,\tau\in[x_1,x_2]$ with $\sigma<\tau$ such that $I_2\subset[\sigma,\tau]\subset I_1$ and $|k-k_u(\sigma,\tau)|> 4\Delta(|k|+\delta(|k|,\eps),\eps_1)$. But $|k-k_u(x_1,x_2)|=0$ (because $k=k_u(x_1,x_2)$) with $I_2\subset[x_1,x_2]=I_1$, so by continuity arguments (see Remark \ref{Continuity-Connected-Arguments}) we can assert that there exist $\bar\sigma,\bar\tau\in[x_1,x_2]$ with $\bar\sigma<\bar\tau$ such that:
\begin{eqnarray}
&&[x^2_1,x^2_2]=I_2\subset[\bar\sigma,\bar\tau]\subset I_1;\label{FundamenTal-AssertION-Which-iS-proVed-by-Induction-EQuATIoN1}
\\
&&\big|k-k_u(\bar\sigma,\bar\tau)\big|=4\Delta\big(|k|+\delta(|k|,\eps),\eps_1\big).\label{FundamenTal-AssertION-Which-iS-proVed-by-Induction-EQuATIoN2}
\end{eqnarray}

\begin{remark}\label{Continuity-Connected-Arguments}
The existence of $\bar\sigma,\bar\tau\in[x_1,x_2]$ with $\bar\sigma<\bar\tau$ satisfying \eqref{FundamenTal-AssertION-Which-iS-proVed-by-Induction-EQuATIoN1} and \eqref{FundamenTal-AssertION-Which-iS-proVed-by-Induction-EQuATIoN2} can be derived as follows. Let $\Psi:[x_1,\sigma]\times[\tau,x_2]\to[0,\infty[$ be defined by $\Psi(y,z):=|k-k_u(y,z)|$. Since $\Psi$ is continuous and $[x_1,\sigma]\times[\tau,x_2]$ is connected, $\Psi([x_1,\sigma]\times[\tau,x_2])$ is an interval. But $0=\Psi(x_1,x_2)\in \Psi([x_1,\sigma]\times[\tau,x_2])$, $\Psi(\sigma,\tau)\in\Psi([x_1,\sigma]\times[\tau,x_2])$ and $0\leq 4\Delta(|k|+\delta(|k|,\eps),\eps_1)<\Psi(\sigma,\tau)$, hence $4\Delta(|k|+\delta(|k|,\eps),\eps_1)\in \Psi([x_1,\sigma]\times[\tau,x_2])$. Consequently, there exist $(\bar\sigma,\bar\tau)\in[x_1,\sigma]\times[\tau,x_2]$ such that $\Psi(\bar\sigma,\bar\tau)=4\Delta(|k|+\delta(|k|,\eps),\eps_1)$, which is the desired conclusion. 
\end{remark}
Set:
\begin{enumerate}
\item[$\bullet$] $A:=\big\{x\in [x_1,x_2]:|u^\prime(x)-k|\geq 2\Delta(|k|+\delta(|k|,\eps),\eps_1)\big\}$;
\item[$\bullet$] $B:=\big\{x\in [\bar\sigma,\bar\tau]:|u^\prime(x)-k_u(\bar\sigma,\bar\tau)|\geq 2\Delta(|k|+\delta(|k|,\eps),\eps_1)\big\}$.
\end{enumerate}
Then 
\begin{equation}\label{FundamenTal-AssertION-Which-iS-proVed-by-Induction-EQuATIoN3}
[\bar\sigma,\bar\tau]\setminus B\subset A. 
\end{equation}
Indeed, if $x\in[\bar\sigma,\bar\tau]\setminus B$ then $x\in I_1=[x_1,x_2]$ by the right inclusion in \eqref{FundamenTal-AssertION-Which-iS-proVed-by-Induction-EQuATIoN1} and $|u^\prime(x)-k_u(\bar\sigma,\bar\tau)|< 2\Delta(|k|+\delta(|k|,\eps),\eps_1)$. But, we have
$$
\big|k-k_u(\bar\sigma,\bar\tau)\big|\leq\big|u^\prime(x)-k\big|+\big|u^\prime(x)-k_u(\bar\sigma,\bar\tau)\big|,
$$
hence, using \eqref{FundamenTal-AssertION-Which-iS-proVed-by-Induction-EQuATIoN2}, 
$$
4\Delta\big(|k|+\delta(|k|,\eps),\eps_1\big)\leq\big|u^\prime(x)-k\big|+2\Delta\big(|k|+\delta(|k|,\eps),\eps_1\big),
$$
and so $|u^\prime(x)-k|\geq2\Delta(|k|+\delta(|k|,\eps),\eps_1)$, which implies that $x\in A$. On the other hand, since $2\Delta\geq \Delta$, $|k|+\delta(|k|,\eps)\geq |k|$ and $\Delta(\cdot,\eps_1)$ is increasing, it is clear that $A\subset \Omega_{x_1,x_2}$, where $\Omega_{x_1,x_2}$ is defined by \eqref{FundaMental-Estimate-Main-Lemma-Set} with $y=x_1$ and $z=x_2$. Recalling that $\eps_1=\eps=|x_1-x_2|$ and using \eqref{FundaMental-Estimate-Main-Lemma} and the fact that $\Delta(\cdot,\eps_1)$ is increasing, we deduce that
\begin{eqnarray}
\int_A\big|u^\prime(x)-k\big|^2dx\leq\int_{\Omega_{x_1,x_2}}\big|u^\prime(x)-k\big|^2dx&\leq&\Delta^2\big(|k|,|x_1-x_2|\big)|x_1-x_2|\nonumber\\
&\leq&\Delta^2\big(|k|+\delta(|k|,\eps),\eps_1\big)|x_1-x_2|.\label{FundamenTal-AssertION-Which-iS-proVed-by-Induction-EQuATIoN4}
\end{eqnarray}
But, by definition of $A$, we have
\begin{equation}\label{FundamenTal-AssertION-Which-iS-proVed-by-Induction-EQuATIoN5}
4\Delta^2\big(|k|+\delta(|k|,\eps),\eps_1\big)|A|\leq \int_A\big|u^\prime(x)-k\big|^2dx,
\end{equation}
and so, combining \eqref{FundamenTal-AssertION-Which-iS-proVed-by-Induction-EQuATIoN4} with \eqref{FundamenTal-AssertION-Which-iS-proVed-by-Induction-EQuATIoN5}, we obtain the following inequality:
\begin{equation}\label{FundamenTal-AssertION-Which-iS-proVed-by-Induction-EQuATIoN6}
|A|\leq{1\over 4}|x_1-x_2|.
\end{equation}
From \eqref{FundamenTal-AssertION-Which-iS-proVed-by-Induction-EQuATIoN2} it is easily seen that $|k_u(\bar\sigma,\bar\tau)|\leq |k|+4\Delta(|k|+\delta(|k|,\eps),\eps_1)$, hence $|k_u(\bar\sigma,\bar\tau)|\leq |k|+\delta(|k|,\eps)$ by \eqref{FundamenTal-AssertION-Which-iS-proVed-by-Induction-EQUaTion-00}, and so $\Delta(|k|+\delta(|k|,\eps),\eps_1)\geq\Delta(|k_u(\bar\sigma,\bar\tau)|,|\bar\sigma-\bar\tau|)$ because $\Delta$ is increasing in both arguments. Thus $B\subset\Omega_{\bar\sigma,\bar\tau}$, where $\Omega_{\bar\sigma,\bar\tau}$ is defined by \eqref{FundaMental-Estimate-Main-Lemma-Set} with $y=\bar\sigma$ and $z=\bar\tau$. Using \eqref{FundaMental-Estimate-Main-Lemma} it follows that
\begin{eqnarray*}
\int_B\big|u^\prime(x)-k_u(\bar\sigma,\bar\tau)\big|^2dx\leq\int_{\Omega_{\bar\sigma,\bar\tau}}\big|u^\prime(x)-k_u(\bar\sigma,\bar\tau)\big|^2dx&\leq&\Delta^2\big(|k_u(\bar\sigma,\bar\tau)|,|\bar\sigma-\bar\tau|\big)|\bar\sigma-\bar\tau|\\
&\leq&\Delta^2\big(|k|+\delta(|k|,\eps),\eps_1\big)|\bar\sigma-\bar\tau|,
\end{eqnarray*}
which combined with that fact that by definition of $B$ we have
$$
4\Delta^2\big(|k|+\delta(|k|,\eps),\eps_1\big)|B|\leq \int_B\big|u^\prime(x)-k_u(\bar\sigma,\bar\tau)\big|^2dx,
$$
gives  the following inequality:
\begin{equation}\label{FundamenTal-AssertION-Which-iS-proVed-by-Induction-EQuATIoN7}
|B|\leq{1\over 4}|\bar\sigma-\bar\tau|.
\end{equation}
From \eqref{FundamenTal-AssertION-Which-iS-proVed-by-Induction-EQuATIoN3} we see that $|\bar\sigma-\bar\tau|\leq |A|+|B|$, hence $|\bar\sigma-\bar\tau|\leq {1\over 4}|x_1-x_2|+{1\over 4}|\bar\sigma-\bar\tau|$ by using \eqref{FundamenTal-AssertION-Which-iS-proVed-by-Induction-EQuATIoN6} and \eqref{FundamenTal-AssertION-Which-iS-proVed-by-Induction-EQuATIoN7}, and so $3|\bar\sigma-\bar\tau|\leq |x_1-x_2|$. On the other hand, from the left inclusion in \eqref{FundamenTal-AssertION-Which-iS-proVed-by-Induction-EQuATIoN1} we see that $|I_2|=|x^2_1-x^2_2|\leq |\bar\sigma-\bar\tau|$, hence ${1\over \exp}|x_1-x_2|\leq  |\bar\sigma-\bar\tau|$, i.e., $|x_1-x_2|\leq\exp|\bar\sigma-\bar\tau|$, by \eqref{Auxiliary-Lemma-For-Main-LeMmA-EQ2}. It follows that $3|\bar\sigma-\bar\tau|\leq \exp|\bar\sigma-\bar\tau|$, which is impossible.

\smallskip

\paragraph{\bf Step 2: induction} Let $m\geq 1$ be such that $\mathcal{P}(m)$ is true. Then, as $I_{m+1}=[x^{m+1}_1,x^{m+1}_2]\subset I_m$ we can apply $\mathcal{P}(m)$ with $[\sigma,\tau]=I_{m+1}$, and we have
$$
\big|k-k_u(x^{m+1}_1,x^{m+1}_2)\big|\leq4\sum_{n=1}^m\Delta\big(|k|+\delta(|k|,\eps),\eps_n\big).
$$
Thus, it is easily seen that for proving that $\mathcal{P}(m+1)$ is true, it is sufficient to establish that 
$$
\forall x_1\hskip-0.5mm\leq \hskip-0.5mm\sigma<\hskip-0.5mm\tau\hskip-0.5mm\leq \hskip-0.5mmx_2\  \hskip-0.5mm\Big[\hskip-0.5mmI_{m+2}\hskip-0.5mm\subset\hskip-0.5mm[\sigma,\tau]\hskip-0.5mm\subset I_{m+1}\hskip-0.5mm\then\hskip-0.5mm|k_u(x^{m+1}_1,x^{m+1}_2)-k_u(\sigma,\tau)|\hskip-0.5mm\leq\hskip-0.5mm 4\Delta(|k|+\delta(|k|,\eps),\eps_{m+1})\hskip-0.5mm\Big].
$$ 
Assume that this latter assertion is false. Then, arguing as in Step 1, we can assert that there exist $\bar\sigma,\bar\tau\in[x^{m+1}_1,x^{m+1}_2]$ with $\bar\sigma<\bar\tau$ such that:
\begin{enumerate}
 \item[$\bullet$] $I_{m+2}\subset[\bar\sigma,\bar\tau]\subset I_{m+1}$; 
  \item[$\bullet$] $\big|k_u(x^{m+1}_1,x^{m+1}_2)-k_u(\bar\sigma,\bar\tau)\big|= 4\Delta(|k|+\delta(|k|,\eps),\eps_{m+1})$.
 \end{enumerate}
Set:
\begin{enumerate}
 \item[$\bullet$] $A_{m+1}:=\big\{x\in [x^{m+1}_1,x^{m+1}_2]:|u^\prime(x)-k_u(x^{m+1}_1,x^{m+1}_2)|\geq 2\Delta(|k_u(x^{m+1}_1,x^{m+1}_2)|+\delta(|k|,\eps),\eps_{m+1})\big\}$; 
  \item[$\bullet$] $B_{m+1}:=\big\{x\in [\bar\sigma,\bar\tau]:|u^\prime(x)-k_u(\bar\sigma,\bar\tau)|\geq 2\Delta(|k_u(x^{m+1}_1,x^{m+1}_2)|+\delta(|k|,\eps),\eps_{m+1})\big\}$.
 \end{enumerate}
 Using the same method as in Step 1, we can establish that $|A_{m+1}|\leq {1\over 4}|x^{m+1}_1-x^{m+1}_2|$, $|B_{m+1}|\leq {1\over 4}|\bar\sigma-\bar\tau|$ and $[\bar\sigma,\bar\tau]\setminus B_{m+1}\subset A_{m+1}$. This latter inclusion implies that $|\bar\sigma-\bar\tau|\leq |A_{m+1}|+|B_{m+1}|$, and so $3|\bar\sigma-\bar\tau|\leq |x^{m+1}_1-x^{m+1}_2|$. On the other hand, as $[x^{m+2}_1,x^{m+2}_2]=I_{m+2}\subset[\bar\sigma,\bar\tau]$ we have $|x^{m+2}_1-x^{m+2}_2|\leq |\bar\sigma-\bar\tau|$, hence ${1\over \exp}|x^{m+1}_1-x^{m+1}_2|\leq  |\bar\sigma-\bar\tau|$, i.e., $|x^{m+1}_1-x^{m+1}_2|\leq\exp|\bar\sigma-\bar\tau|$, by \eqref{Auxiliary-Lemma-For-Main-LeMmA-EQ2}. It follows that $3|\bar\sigma-\bar\tau|\leq \exp|\bar\sigma-\bar\tau|$, which is impossible.
\endproof

\medskip

\section{Proof of the general regularity theorem}

In this section we prove Theorem \ref{General-Regularity-Theorem} (see \S 4.2). 

\subsection{Auxiliary lemmas} To prove Theorem \ref{General-Regularity-Theorem} we need the following lemmas.

\begin{lemma}\label{Lemma-MRT-L1}
Under the hypotheses of Theorem {\rm\ref{General-Regularity-Theorem}}, if $u\in\mathcal{L}_\omega(L,K,c,\delta_0)$ has a finite derived number $d\in\RR$ at $\widebar{x}\in]a,b[$, i.e., there exist two sequences $\{s_n\}_{n\geq 1}$ and $\{t_n\}_{n\geq 1}$ with $s_n<\widebar{x}<t_n$ such that{\rm:}
\begin{eqnarray}
&& \lim_{n\to\infty} s_n=\lim_{n\to\infty} t_n=\widebar{x}; \label{Lemma-MRT-L1-EqUatION-1}\\
&& \lim_{n\to\infty} k_u(s_n,t_n)=d,\label{Lemma-MRT-L1-EqUatION-2}
\end{eqnarray}
then $u$ is differentiable at $\widebar{x}$ and $u^\prime(x)=d$. 
\end{lemma}

\newtheorem*{Lemma-MRT-L1-Bis}{\bf Lemma \ref{Lemma-MRT-L1}-bis}

\begin{Lemma-MRT-L1-Bis}
{\em Under the hypotheses of Theorem {\rm\ref{General-Regularity-Theorem}}, if $u\in\mathcal{L}_\omega(L,K,c,\delta_0)$ has a finite derived number $d\in\RR$ at $a$ (resp. $b$), i.e., there exists a sequence  $\{t_n\}_{n\geq 1}$ (resp. $\{s_n\}_{n\geq 1}$) with $\widebar{x}<t_n$ (resp. $s_n<\widebar{x}$) such that{\rm:}
\begin{eqnarray}
&& \lim_{n\to\infty} t_n=\widebar{x}\quad\hbox{(resp. } \lim_{n\to\infty} s_n=\widebar{x}\hbox{)}; \label{Lemma-MRT-L1-EqUatION-1-BiS}\\
&& \lim_{n\to\infty} k_u(a,t_n)=d\quad\hbox{(resp. } \lim_{n\to\infty} k_u(s_n,b)=d\hbox{)},\label{Lemma-MRT-L1-EqUatION-2-BiS}
\end{eqnarray}
then $u$ is differentiable at $a$ (resp. $b$) and $u^\prime(a)=d$ (resp. $u^\prime(a)=d$).}
\end{Lemma-MRT-L1-Bis}

\begin{proof}[\bf Proof of Lemma \ref{Lemma-MRT-L1}]
By \eqref{Lemma-MRT-L1-EqUatION-2} there exists $M>0$ such that $|k_u(s_n,t_n)|\leq M$ for all $n\geq 1$. Setting $\eps_n:=|s_n-t_n|$ for all $n\geq 1$ (where, because of \eqref{Lemma-MRT-L1-EqUatION-1}, without loss of generality we can assume  that $\eps_n\leq{\delta\over\exp}$ for all $n\geq 1$) from Lemma \ref{Main-Lemma}, see \eqref{Regu-Fond-Inequality}, we have
$$
\left|k_u(s,t)-k_u(s_n,t_n)\right|\leq\delta(M,\eps_n)
$$
for all $s,t\in[s_n,t_n]$ with $s<t$ and all $n\geq 1$ (where  $\delta$ is given by Lemma \ref{Main-Lemma}). Thus
\begin{eqnarray}
\left|k_u(s,t)-d\right|&\leq& \left|k_u(s,t)-k_u(s_n,t_n)\right|+\left|k_u(s_n,t_n)-d\right|\nonumber\\
&\leq&\delta(M,\eps_n)+\left|k_u(s_n,t_n)-d\right|\label{Lemma-MRT-L1-EqUatION-3}
\end{eqnarray}
for all $s,t\in[s_n,t_n]$ with $s<t$ and all $n\geq 1$. But $\lim_{n\to\infty}\eps_n=0$ by \eqref{Lemma-MRT-L1-EqUatION-1} and so $\lim_{n\to\infty}\delta(M,\eps_n)=0$ by \eqref{Fundamental-Limit}. Moreover $\lim_{n\to\infty}|k_u(s_n,t_n)-d|=0$ by \eqref{Lemma-MRT-L1-EqUatION-2}. So, from \eqref{Lemma-MRT-L1-EqUatION-3} we  can deduce that
$$
\lim_{\Tiny\begin{array}{c} s,t\to\widebar{x}\\ s<\widebar{x}<t\end{array}}k_u(s,t)=d,
$$
which proves that $u$ is differentiable at $\widebar{x}$ and $u^\prime(\widebar{x})=d$.
\end{proof}

\smallskip

\begin{proof}[\bf Proof of Lemma \ref{Lemma-MRT-L1}-bis]
This follows by similar arguments as in the proof of Lemma \ref{Lemma-MRT-L1} by  using \eqref{Lemma-MRT-L1-EqUatION-1-BiS} and \eqref{Lemma-MRT-L1-EqUatION-2-BiS} instead of \eqref{Lemma-MRT-L1-EqUatION-1} and \eqref{Lemma-MRT-L1-EqUatION-2}.
\end{proof}

\begin{lemma}\label{Lemma-MRT-L2}
Under the hypotheses of Theorem {\rm\ref{General-Regularity-Theorem}}, for each $C>0$ and each $\eta>0$ there exists $\widehat{\delta}(C,\eta)>0$ such that for every $u\in \mathcal{L}_\omega(L,K,c,\delta_0)$, one has 
$$
\Big[|u^\prime(\widebar{x})|<C\hbox{ and }|x-\widebar{x}|<\widehat{\delta}(C,\eta)\Big]\then |u^\prime(x)-u^\prime(\widebar{x})|\leq\eta.
$$
\end{lemma}

\begin{proof}[\bf Proof of Lemma \ref{Lemma-MRT-L2}]
Let $C>0$ and $\eta>0$. By \eqref{Fundamental-Limit} we have
$$
\lim_{\eps\to 0}\delta(C+\eta,\eps)=0.
$$
So, we can assert that there exists $\eps_{C,\eta}\in]0,{\delta_0\over\exp}]$ such that
\begin{equation}\label{EQuATiOn-LeMMa-2-Eq0}
\delta\big(C+\eta,\eps_{C,\eta}\big)<{\eta\over 2}.
\end{equation}
Set $\widehat{\delta}(C,\eta):=\eps_{C,\eta}$ and fix $u\in \mathcal{L}_\omega(L,K,c,\delta_0)$. Let $x,\widebar{x}\in[a,b]$ be such that $|u^\prime(\widebar{x})|<C$ and $|x-\widebar{x}|<\widehat{\delta}(C,\eta)$ with, without loss of generality, $x<\widebar{x}$. We claim that
\begin{equation}\label{EQuATiOn-LeMMa-2-Eq0-1}
\big|k_u(x,\widebar{x})-u^\prime(\widebar{x})\big|\leq{\eta\over 2}.
\end{equation}
Indeed, otherwise we have
$
|k_u(x,\widebar{x})-u^\prime(\widebar{x})|>{\eta\over 2}.
$
By definition of $u^\prime(\widebar{x})$, there exists $x_1\in]x,\widebar{x}[$ such that 
\begin{equation}\label{EQuATiOn-LeMMa-2-Eq1}
\big|k_u(y,\widebar{x})-u^\prime(\widebar{x})\big|<{\eta\over 2}\hbox{ for all }y\in[x_1,\widebar{x}[. 
\end{equation}
Set:
\begin{enumerate}
\item[$\bullet$] $A:=\big\{y\in[x,\widebar{x}[:|k_u(y,\widebar{x})-u^\prime(\widebar{x})|<{\eta\over 2}\big\}$;
\item[$\bullet$] $B:=\big\{y\in[x,\widebar{x}[:|k_u(y,\widebar{x})-u^\prime(\widebar{x})|>{\eta\over 2}\big\}$.
\end{enumerate}
Then, from the above, $A\not=\emptyset$, $B\not=\emptyset$ and $A\cap B=\emptyset$. Moreover, both $A$ and $B$ are open in $[x,\widebar{x}[$. As $[x,\widebar{x}[$ is a connected set it follows that $[x,\widebar{x}[\setminus(A\cup B)\not=\emptyset$, and consequently there exists $\widebar{y}\in[x,\widebar{x}[$ such that
\begin{equation}\label{EQuATiOn-LeMMa-2-Eq2}
\big|k_u(\widebar{y},\widebar{x})-u^\prime(\widebar{x})\big|={\eta\over 2}.
\end{equation}
From Lemma \ref{Main-Lemma}, see \eqref{Regu-Fond-Inequality}, and the fact that $|\widebar{y}-\widebar{x}|\leq|x-\widebar{x}|<\widehat{\delta}(C,\eta)$ with $\widehat{\delta}(C,\eta)\leq{\delta_0\over\exp}$, we have
$$
\big|k_u(\widebar{y},\widebar{x})-k_u(y,\widebar{x})\big|\leq\delta\big(|k_u(\widebar{y},\widebar{x})|,|\widebar{y}-\widebar{x}|\big)\leq \delta\big(|k_u(\widebar{y},\widebar{x})|,\widehat{\delta}(C,\eta)\big)
$$
for all $y\in[\widebar{y},\widebar{x}[$ (where $\delta$ is given by Lemma \ref{Main-Lemma}), and letting $y\to\widebar{x}$ we obtain
$$
\big|k_u(\widebar{y},\widebar{x})-u^\prime(\widebar{x})\big|\leq\delta\big(|k_u(\widebar{y},\widebar{x})|,\widehat{\delta}(C,\eta)\big).
$$
But $|k_u(\widebar{y},\widebar{x})|\leq |u^\prime(\widebar{x})|+{\eta\over 2}<C+{\eta\over 2}<C+\eta$ by \eqref{EQuATiOn-LeMMa-2-Eq2},  hence
$$
\big|k_u(\widebar{y},\widebar{x})-u^\prime(\widebar{x})\big|\leq\delta\big(C+\eta,\widehat{\delta}(C,\eta)\big)
$$
because $\delta$ is increasing in both arguments. According to \eqref{EQuATiOn-LeMMa-2-Eq0} and \eqref{EQuATiOn-LeMMa-2-Eq2} we see that
$$
{\eta\over 2}=\big|k_u(\widebar{y},\widebar{x})-u^\prime(\widebar{x})\big|\leq\delta\big(C+\eta,\widehat{\delta}(C,\eta)\big)<{\eta\over 2},
$$
which is impossible. So, \eqref{EQuATiOn-LeMMa-2-Eq0-1} is proved. On the other hand, since $|k_u(x,\widebar{x})|<C+{\eta\over 2}<C+\eta$ by \eqref{EQuATiOn-LeMMa-2-Eq0-1}, and $|x-\widebar{x}|<\widehat{\delta}(C,\eta)$, taking \eqref{EQuATiOn-LeMMa-2-Eq0} into account and using again Lemma \ref{Main-Lemma} we have
\begin{equation}\label{EQuATiOn-LeMMa-2-Eq0-2}
\big|k_u(x,y)-k_u(x,\widebar{x})\big|\leq \delta\big(|k_u(x,\widebar{x})|,|x-\widebar{x}|\big)\leq\delta\big(C+\eta,\widehat{\delta}(C,\eta))<{\eta\over 2}\hbox{ for all }y\in]x,\widebar{x}].
\end{equation}
\begin{remark}
Since $|k_u(x,\widebar{x})|<C+{\eta\over 2}$, from \eqref{EQuATiOn-LeMMa-2-Eq0-2} it follows that $|k_u(x,y)|<C+\eta$ for all $y\in]x,\widebar{x}]$, which implies that $u$ has a finite derived number at $x$ in the sense of lemmas \ref{Lemma-MRT-L1} and \ref{Lemma-MRT-L1}-bis. Hence $u$ is differentiable at $x$.
\end{remark}
Letting $y\to x$ in \eqref{EQuATiOn-LeMMa-2-Eq0-2} we obtain
\begin{equation}\label{EQuATiOn-LeMMa-2-Eq0-1-1}
\big|u^\prime(x)-k_u(x,\widebar{x})\big|\leq {\eta\over 2}.
\end{equation}
From \eqref{EQuATiOn-LeMMa-2-Eq0-1} and \eqref{EQuATiOn-LeMMa-2-Eq0-1-1} we conclude that
$$
\big|u^\prime(x)-u^\prime(\widebar{x})\big|\leq\big|u^\prime(x)-k_u(x,\widebar{x})\big|+\big|k_u(x,\widebar{x})-u^\prime(\widebar{x})\big|\leq {\eta\over 2}+ {\eta\over 2}=\eta,
$$
and the proof is complete.
\end{proof}

\subsection{Proof of Theorem \ref{General-Regularity-Theorem}}

Let $u\in\mathcal{L}_\omega(L,K,c,\delta_0)$. Following Definition \ref{TPR-Def}, we have to prove the following assertions:
\begin{enumerate}
\item[{\rm (a)}] if $x_0\in\Omega_u$ then there exists a neighborhood $V_{x_0}$ of $x_0$ such that $V_{x_0}\subset\Omega_u$;
\item[{\rm (b)}] if $x_0\in\Omega_u$ then $\lim\limits_{x\to x_0} u^\prime(x)=u^\prime(x_0)$; 
\item[{\rm (c)}] if $x_0\not\in \Omega_u$ then $u^\prime(x_0)\in\{-\infty,\infty\}$;
\item[{\rm (d)}] if $u^\prime(x_0)=\infty$ (resp. $u^\prime(x_0)=-\infty$) then $\lim\limits_{x\to x_0} u^\prime(x)=\infty$ (resp. $\lim\limits_{x\to x_0} u^\prime(x)=\infty$),
\end{enumerate}
where $\Omega_u:=\big\{x\in[a,b]:u\hbox{ is differentiable at }x\big\}$.

\smallskip

\paragraph{\bf Proof of (a)} Let $x_0\in\Omega_u$. Without loss of generality we can assume that $x_0\in]a,b[$. (If $x_0=a$ (resp. $x_0=b$) the proof will follow by similar arguments by using 
$$
\lim_{\Tiny\begin{array}{c} x_2\to a\\ a<x_2\end{array}}k_u(a,x_2)=u^\prime(a)\quad \Big(\hbox{resp. }\lim_{\Tiny\begin{array}{c} x_1\to b\\ x_1<b\end{array}}k_u(x_1,b)=u^\prime(b)\Big)
$$
instead of \eqref{Eq-Proof-of-(a)-Equat-1} and Lemma \ref{Lemma-MRT-L1}-bis instead of Lemma \ref{Lemma-MRT-L1}.) Then
\begin{equation}\label{Eq-Proof-of-(a)-Equat-1}
\lim_{\Tiny\begin{array}{c} x_1,x_2\to x_0\\ x_1<x_0<x_2\end{array}}k_u(x_1,x_2)=u^\prime(x_0).
\end{equation}
From \eqref{Eq-Proof-of-(a)-Equat-1} we can assert that there exists $C>0$ such that
$
\big|k_u(x_1,x_2)\big|\leq C
$
for all $x_1,x_2\in V_{x_0}$ with $x_1<x_0<x_2$, where $V_{x_0}\subset]x_0-{\delta_0\over 2\exp},x_0+{\delta_0\over 2\exp}[$ is a neighborhood of $x_0$. Using Lemma \ref{Main-Lemma}, see \eqref{Regu-Fond-Inequality}, it follows that for every $x_1,x_2\in V_{x_0}$ with $x_1<x_0<x_2$ and every $s,t\in[x_1,x_2]$ with $s<t$, we have
$$
\left|k_u(x_1,x_2)-k_u(s,t)\right|\leq\delta\big(\left|k_u(x_1,x_2)\right|,|x_1-x_2|\big)\leq\delta(C,|x_1-x_2|)
$$
(with $\delta$ given by Lemma \ref{Main-Lemma}) which implies that
$$
\left|k_u(s,t)\right|\leq\delta(C,|x_1-x_2|)+\left|k_u(x_1,x_2)\right|\leq M
$$
with $M:=\delta(C,|x_1-x_2|)+C>0$. Fix any $\widebar{x}\in V_{x_0}$. Then, there is a neighborhood $V_{\widebar{x}}\subset V_{x_0}$ of $\widebar{x}$ such that $\left|k_u(s,t)\right|\leq M$ for all $s,t\in V_{\widebar{x}}$ with $s<\widebar{x}<t$, hence 
$$
d:=\liminf_{\Tiny\begin{array}{c} s,t\to\widebar{x}\\ s<\widebar{x}<t\end{array}}k_u(s,t)\in\RR,
$$
and consequently there exist two sequences $\{s_n\}_{n\geq 1}$ and $\{t_n\}_{n\geq 1}$ with $s_n<\widebar{x}<t_n$ verifying \eqref{Lemma-MRT-L1-EqUatION-1} and \eqref{Lemma-MRT-L1-EqUatION-2}, which implies that $u$ is differentiable at $\widebar{x}$ (and $u^\prime(\widebar{x})=d$), i.e., $\widebar{x}\in\Omega_u$, by using Lemma \ref{Lemma-MRT-L1}. \endproof

\smallskip

\paragraph{\bf Proof of (b)} Let $x_0\in\Omega_u$. By (a) there exists a neighborhood $V_{x_0}$ of $x_0$ such that $V_{x_0}\subset\Omega_u$. Without loss of generality we can assume that $x_0\in]a,b[$. (If $x_0=a$ (resp. $x_0=b$) the proof will be the same by considering that we have one sequence $\{x^n_2\}_{n\geq 1}$ (resp. $\{x^n_1\}_{n\geq 1}$) with $a<x^n_2$ (resp. $x^n_1<b$) instead of $x^n_1<x_0<x^n_2$ and by replacing ``$k_u(x^n_1,x^n_2)$" by ``$k_u(a,x^n_2)$" (resp. ``$k_u(x^n_1,b)$").) Let $\{x_n\}_{n\geq 1}\subset V_{x_0}$ be such that 
\begin{equation}\label{EquAtion-Proof-OF-(b)-Number1}
\lim_{n\to\infty}x_n=x_0.
\end{equation}
As $x_0\in\Omega_u$, taking \eqref{EquAtion-Proof-OF-(b)-Number1} into account, we can assert that there exist two sequences $\{x^n_1\}_{n\geq 1}$ and $\{x^n_2\}_{n\geq 1}$ with $x^n_1<x_0<x^n_2$  such that:
\begin{eqnarray}
&& x^n_1<x_n<x^n_2 \hbox{ for all }n\geq 1; \label{EquAtion-Proof-OF-(b)-Number2}\\
&& \lim_{n\to\infty} x^n_1=\lim_{n\to\infty} x^n_2=x_0; \label{EquAtion-Proof-OF-(b)-Number3}\\
&& \lim_{n\to\infty} k_u(x^n_1,x^n_2)=u^\prime(x_0).\label{EquAtion-Proof-OF-(b)-Number4}
\end{eqnarray}
(Because of \eqref{EquAtion-Proof-OF-(b)-Number3} without loss of generality we can assume that $|x^n_1-x^n_2|\leq{\delta_0\over\exp}$ for all $n\geq 1$.) On the other hand, we have $x_n\in\Omega_u$ for all $n\geq 1$. Hence, taking \eqref{EquAtion-Proof-OF-(b)-Number2} into account, for each $n\geq 1$ there exist two sequences $\{s^n_j\}_{j\geq 1}$ and $\{t^n_j\}_{j\geq 1}$ with $s^n_j<x_n<t^n_j$ such that:
\begin{eqnarray}
&& s^n_j,t^n_j\in[x^n_1,x^n_2] \hbox{ for all }j\geq 1; \nonumber\\
&& \lim_{j\to\infty} s^n_j=\lim_{j\to\infty} t^n_j=x_n; \nonumber\\
&& \lim_{j\to\infty} k_u(s^n_j,t^n_j)=u^\prime(x_n).\label{EquAtion-Proof-OF-(b)-Number7}
\end{eqnarray}
Let $M>0$ be such that $|k_u(x^n_1,x^n_2)|\leq M$ for all $n\geq 1$ (such a positive constant $M$ exists because of \eqref{EquAtion-Proof-OF-(b)-Number4}). Using Lemma \ref{Main-Lemma}, see \eqref{Regu-Fond-Inequality}, we see that for every $n\geq 1$ and every $j\geq 1$,
\begin{eqnarray*}
\big|u^\prime(x_n)-u^\prime(x_0)\big|&=&\big|u^\prime(x_n)-k_u(s^n_j,t^n_j)+k_u(s^n_j,t^n_j)-k_u(x^n_1,x^n_2)+k_u(x^n_1,x^n_2)-u^\prime(x_0)\big|\\
&\leq&\big|u^\prime(x_n)-k_u(s^n_j,t^n_j)\big|+\big|k_u(s^n_j,t^n_j)-k_u(x^n_1,x^n_2)\big|+\big|k_u(x^n_1,x^n_2)-u^\prime(x_0)\big|\\
&\leq&\big|u^\prime(x_n)-k_u(s^n_j,t^n_j)\big|+\delta\big(|k_u(x^n_1,x^n_2)|,|x^n_1-x^n_2|\big)+\big|k_u(x^n_1,x^n_2)-u^\prime(x_0)\big|\\
&\leq&\big|u^\prime(x_n)-k_u(s^n_j,t^n_j)\big|+\delta\big(M,|x^n_1-x^n_2|\big)+\big|k_u(x^n_1,x^n_2)-u^\prime(x_0)\big|
\end{eqnarray*}
(with $\delta$ given by Lemma \ref{Main-Lemma}). Letting $j\to\infty$ and using \eqref{EquAtion-Proof-OF-(b)-Number7} and then letting $n\to\infty$ and using \eqref{EquAtion-Proof-OF-(b)-Number4} and \eqref{EquAtion-Proof-OF-(b)-Number3} together with \eqref{Fundamental-Limit}, we conclude that $\lim\limits_{n\to\infty}u^\prime(x_n)=u^\prime(x_0)$.
\endproof

\smallskip

\paragraph{\bf Proof of (c)} Let $x_0\not\in\Omega_u$. Without loss of generality we can assume that $x_0\in]a,b[$. (If $x_0=a$ (resp. $x_0=b$) the proof will follows by similar arguments by using 
$$
\liminf_{\Tiny\begin{array}{c} t\to a\\ a<t\end{array}}|k_u(a,t)|=\infty\quad \Big(\hbox{resp. }\liminf_{\Tiny\begin{array}{c} s\to b\\ s<b\end{array}}|k_u(s,b)|=\infty\Big)
$$
instead of \eqref{Eq-Proof-of-(a)-Equat-1-bis} and Lemma \ref{Lemma-MRT-L1}-bis instead of Lemma \ref{Lemma-MRT-L1}.) Then
\begin{equation}\label{Eq-Proof-of-(a)-Equat-1-bis}
\liminf_{\Tiny\begin{array}{c} s,t\to{x_0}\\ s<{x_0}<t\end{array}}|k_u(s,t)|=\infty.
\end{equation}
Indeed, if \eqref{Eq-Proof-of-(a)-Equat-1-bis} is false, i.e.,  
$$
\alpha:=\liminf_{\Tiny\begin{array}{c} s,t\to{x_0}\\ s<{x_0}<t\end{array}}|k_u(s,t)|<\infty,
$$
then there exist two sequences $\{s_j\}_{j\geq 1}$ and $\{t_j\}_{j\geq 1}$ with $s_j<x_0<t_j$ such that:
\begin{eqnarray}
&& \lim_{j\to\infty} s_j=\lim_{j\to\infty} t_j=x_0; \label{Lemma-MRT-L1-EqUatION-1-bis-bis}\\
&& \lim_{j\to\infty} |k_u(s_j,t_j)|=\alpha.\label{Lemma-MRT-L1-EqUatION-2-bis-bis}
\end{eqnarray}
By \eqref{Lemma-MRT-L1-EqUatION-2-bis-bis} there exists $M>0$ such that $|k_u(s_j,t_j)|\leq M$ for all $j\geq 1$, and so from Bolzano-Weierstrass's theorem it follows that there exist $d\in[-M,M]$ and two subsequences $\{s_{j_n}\}_{n\geq 1}$ and $\{t_{j_n}\}_{n\geq 1}$ of $\{s_j\}_{j\geq 1}$ and $\{t_j\}_{j\geq 1}$ respectively such that
$$
\lim_{n\to\infty} k_u(s_{j_n},t_{j_n})=d.
$$
But, taking \eqref{Lemma-MRT-L1-EqUatION-1-bis-bis} into account, we also have ($s_{j_n}<x_0<t_{j_n}$ for all $n\geq 1$ and)
$$
\lim_{n\to\infty} s_{j_n}=\lim_{n\to\infty} t_{j_n}=x_0,
$$
hence, using Lemma \ref{Lemma-MRT-L1}, we can assert that $u$ is differentiable at $x_0$ (and $u^\prime(x_0)=d)$, i.e., $x_0\in\Omega_u$, which gives a contradiction. Thus \eqref{Eq-Proof-of-(a)-Equat-1-bis} is true and so
$$
\lim_{\Tiny\begin{array}{c} s,t\to{x_0}\\ s<{x_0}<t\end{array}}|k_u(s,t)|=\infty,
$$
which implies that
$$
\lim_{\Tiny\begin{array}{c} s,t\to{x_0}\\ s<{x_0}<t\end{array}}k_u(s,t)=-\infty\quad\hbox{ or } \lim_{\Tiny\begin{array}{c} s,t\to{x_0}\\ s<{x_0}<t\end{array}}k_u(s,t)=\infty,
$$
i.e., $u^\prime(x_0)\in\{-\infty,\infty\}$. \endproof

\smallskip

\paragraph{\bf Proof of (d)} Let $x_0\in[a,b]$ be such that 
\begin{equation}\label{Assumption-Proof-OF-(d)}
u^\prime(x_0)=\infty. 
\end{equation}
We have to prove that
\begin{equation}\label{Equation-Proof-Of-(d)-Eq1}
\lim_{x\to x_0}u^\prime(x)=\infty.
\end{equation}
Assume that \eqref{Equation-Proof-Of-(d)-Eq1} is false. Then, there exist $M>0$ and a sequence $\{x_n\}_{n\geq 1}$ with, without loss of generality, $x_n<x_0$ such that: 
\begin{eqnarray}
&&\lim_{n\to\infty} x_n=x_0;\label{Equation-Proof-Of-(d)-Eq2}\\
&&u^\prime(x_n)<M\hbox{ for all }n\geq 1.\label{Equation-Proof-Of-(d)-Eq3}
\end{eqnarray}
From \eqref{Equation-Proof-Of-(d)-Eq3} we can assert that for every $n\geq 1$ there exists a neighborhood $V_{x_n}$ of $x_n$ such that
\begin{equation}\label{Equation-Proof-Of-(d)-Eq4}
k_u(x_n,y)<M
\end{equation}
for all $y\in V_{x_n}\cap]x_n,x_0]$. Taking \eqref{Equation-Proof-Of-(d)-Eq2} into account, since $u^\prime(x_0)=\infty$ we have
$$
\lim_{n\to\infty}k_u(x_n,x_0)=\infty,
$$
and so
\begin{equation}\label{Equation-Proof-Of-(d)-Eq5}
k_u(x_n,x_0)>M+1
\end{equation}
for all $n\geq n_0$ with $n_0\geq 1$ sufficiently large. On the other hand, as $u$ is continuous at $x_0$, for each $n\geq 1$ we have
\begin{equation}\label{Equation-Proof-Of-(d)-Eq6}
\lim_{\Tiny\begin{array}{c} y\to{x_0}\\ x_n<y<{x_0}\end{array}}k_u(x_n,y)=k_u(x_n,x_0).
\end{equation}
Fix any $n\geq n_0$. From \eqref{Equation-Proof-Of-(d)-Eq5} and \eqref{Equation-Proof-Of-(d)-Eq6} we deduce that  there exists a neighborhood $V_{x_0}$ of $x_0$ such that
\begin{equation}\label{Equation-Proof-Of-(d)-Eq7}
k_u(x_n,y)>M
\end{equation}
for all $y\in V_{x_0}\cap]x_n,x_0]$. Then $A_n\cap B_n=\emptyset$ with:
\begin{enumerate}
\item[$\bullet$] $A_n:=\big\{y\in]x_n,x_0]:\eqref{Equation-Proof-Of-(d)-Eq4} \hbox{ holds}\big\}$;
\item[$\bullet$] $B_n:=\big\{y\in]x_n,x_0]:\eqref{Equation-Proof-Of-(d)-Eq7} \hbox{ holds}\big\}$.
\end{enumerate} 
But, from the above, $A_n\not=\emptyset$ and  $B_n\not=\emptyset$, and both $A_n$ and $B_n$ are open in $]x_n,x_0]$. As $]x_n,x_0]$ is a connected set we can assert that $]x_n,x_0]\setminus(A_n\cup B_n)\not=\emptyset$. Consequently, for each $n\geq n_0$ there exists $y_n\in]x_n,x_0]$ such that
\begin{equation}\label{Equation-Proof-Of-(d)-Eq8}
k_u(x_n,y_n)=M\hbox{, i.e., }\big|k_u(x_n,y_n)\big|=M.
\end{equation}
Moreover, by \eqref{Equation-Proof-Of-(d)-Eq2} we see that $\lim_{n\to\infty}y_n=x_0$ and so
\begin{equation}\label{Equation-Proof-Of-(d)-Eq9}
\lim_{n\to\infty}|x_n-y_n|=0.
\end{equation}
(Thus, without loss of generality we can assume that $|x_n-y_n|\leq{\delta_0\over\exp}$ for all $n\geq 1$.) But, for any $n\geq n_0$, from Lemma \ref{Main-Lemma}, see \eqref{Regu-Fond-Inequality}, we have
$$
\big|k_u(x_n,y)-k_u(x_n,y_n)\big|\leq\delta\big(|k_u(x_n,y_n)|,|x_n-y_n|\big)
$$
for all $y\in]x_n,y_n]$ (where $\delta$ is given by Lemma \ref{Main-Lemma}), hence
$$
\big|k_u(x_n,y)-k_u(x_n,y_n)\big|\leq\delta\big(M,|x_n-y_n|\big)
$$
by using \eqref{Equation-Proof-Of-(d)-Eq8}. Letting $y\to x_n$ it follows that
$$
\big|u^\prime(x_n)-k_u(x_n,y_n)\big|\leq\delta\big(M,|x_n-y_n|\big)
$$
for all $n\geq n_0$. Using  \eqref{Equation-Proof-Of-(d)-Eq9} and taking \eqref{Fundamental-Limit} into account, there exists $n_1\geq n_0$ such that
$$
\big|u^\prime(x_n)-k_u(x_n,y_n)\big|< 1\hbox{ for all }n\geq n_1.
$$
Thus
\begin{equation}\label{Pre-Final-Eq-Proof-of-(d)}
\big|u^\prime(x_n)\big|< M+1\hbox{ for all }n\geq n_1.
\end{equation}
Let $\widehat{\delta}(M+1,1)>0$ be given by Lemma \ref{Lemma-MRT-L2} (with $C=M+1$ and $\eta=1$). Then, we have
\begin{equation}\label{Final-Eq-Proof-of-(d)}
\Big[|u^\prime(\widebar{x})|<M+1\hbox{ and }|x-\widebar{x}|<\widehat{\delta}(M+1,1)\Big]\then |u^\prime(x)-u^\prime(\widebar{x})|\leq1.
\end{equation}
By \eqref{Equation-Proof-Of-(d)-Eq2} there exists $n_2\geq n_1$ such that $|x_0-x_{n_2}|<\widehat{\delta}(M+1,1)$. Hence, taking \eqref{Pre-Final-Eq-Proof-of-(d)} into account and  using \eqref{Final-Eq-Proof-of-(d)} (with $x=x_0$ and $\widebar{x}=x_{n_2}$), we obtain  $|u^\prime(x_0)-u^\prime(x_{n_2})|\leq1$. Consequently, using again \eqref{Pre-Final-Eq-Proof-of-(d)}, i.e., $|u^\prime(x_{n_2})|<M+1$, we deduce that
$$
\big|u^\prime(x_0)\big|\leq \big|u^\prime(x_0)-u^\prime(x_{n_2})\big|+\big|u^\prime(x_{n_2})\big|<M+2,
$$
which contradicts \eqref{Assumption-Proof-OF-(d)}. (Similarly, we can prove that if $u^\prime(x_0)=-\infty$ then $\lim_{x\to x_0}u^\prime(x)=-\infty$.) \endproof

\bigskip

\paragraph{\bf Acknowledgement} I gratefully acknowledge  M. A. Sychev for introducing me to the subject of regularity of one dimensional variational obstacle problems, and for his many comments during the preparation of this paper.

\end{document}